\documentclass[a4paper,11pt]{article}

\usepackage[T1]{fontenc}
\usepackage{lmodern}
\usepackage{eurosym}
\usepackage{lastpage}
\usepackage{xspace}
\usepackage[margin=15mm,includehead,includefoot]{geometry}
\usepackage{fancyhdr}
\usepackage{booktabs}
\usepackage{graphicx}
\usepackage{multirow}
\usepackage{array}
\usepackage{xcolor}
\usepackage{csquotes}
\usepackage{pgfgantt}
\usepackage{titlesec}
\usepackage{hyperref}
\usepackage{comment}
\usepackage{amsthm}
\usepackage{amsmath}
\usepackage{amssymb}
\usepackage{pifont}
\usepackage{tabto}
\usepackage{enumitem}

\usepackage{algorithm}
\usepackage{algpseudocode}

\usepackage{graphicx}
\usepackage{floatflt}

\titleformat*{\section}{\large\bfseries}
\titleformat*{\subsection}{\normalsize\bfseries}
\titleformat*{\subsubsection}{\normalsize\bfseries}

\renewcommand{\footnotesize}{\fontsize{8bp}{1em}\selectfont}

\newtheorem{theorem}{Theorem}[section]
\newtheorem{proposition}[theorem]{Proposition}

\newtheorem{definition}[theorem]{Definition}
\newtheorem{example}[theorem]{Example}

\newtheorem{remark}[theorem]{Remark}

\newtheorem{corollary}[theorem]{Corollary}
\newtheorem{openproblem}[theorem]{Open Problem}

\usepackage{authblk}

\usepackage{setspace}
\begin{document}

\title{\bf Market share maximizing strategies of CAV fleet operators may cause chaos in our cities}
\author[1]{Grzegorz Jamr\'oz$^*$}
\author[1]{Rafa{\l}  Kucharski}
\author[2]{David Watling}
\affil[1]{Faculty of Mathematics and Computer Science, Jagiellonian University, Kraków, Poland}
\affil[2]{Institute for Transport Studies, University of Leeds, United Kingdom}
\affil[*]{Corresponding author, e-mail: grzegorz.jamroz@uj.edu.pl}
\setcounter{Maxaffil}{0}
\renewcommand\Affilfont{\itshape\small}
\maketitle

\begin{abstract}
We study the dynamics and equilibria of a new kind of routing games, where players -- drivers of future autonomous vehicles -- may switch between individual (HDV) and collective (CAV) routing. In individual routing, just like today, drivers select routes minimizing expected travel costs, whereas in collective routing an operator centrally assigns vehicles to routes. The utility is then the average experienced travel time discounted with individually perceived attractiveness of automated driving. The market share maximising strategy amounts to offering utility greater than for individual routing to as many drivers as possible. Our theoretical contribution consists in developing a rigorous mathematical framework of individualized collective routing and studying algorithms which fleets of CAVs may use for their market-share optimization. We also define bi-level CAV -- HDV equilibria and derive conditions which link the potential marketing behaviour of CAVs to the behavioural profile of the human population. Importantly, our general framework admits further extensions and promises fruitful theoretical and applied further research on the topic. Practically, we find that the fleet operator may often be able to equilibrate at full market share by simply mimicking the choices HDVs would make. In more realistic  heterogenous human population settings, however, we discover that the market-share maximizing fleet controller should use highly variable mixed strategies as a means to attract or retain customers. The reason is that in mixed routing the powerful group player can control which vehicles are routed via congested and uncongested alternatives. The congestion pattern generated by CAVs is, however, not known to HDVs before departure and so HDVs cannot select faster routes and face huge uncertainty whichever alternative they choose. Consequently, mixed market-share maximising fleet strategies resulting in unpredictable day-to-day driving conditions may, alarmingly, become pervasive in our future cities.

\end{abstract}


\section{Introduction}
Day-to-day travel times in an equilibrated traffic system are not fixed but stochastic, \cite{Taylor1982,Mahmassani1,WatlingCantarella, EverChanging}, which is due to \cite{Mahmassani1} traffic incidents, special events, work zones, weather, day-to-day demand fluctuations, traffic control devices, or inadequate base capacity. The typical distribution of day-to-day travel times can be efficiently modeled as a log-normal or Burr distribution \cite{Buchel, Susilawati, Taylor} with a well-defined peak and a long positive tail. 
Focusing on the demand-induced variability of travel times, we note that the magnitude of travel time variability in low congestion regime is limited, however it grows significantly when congestion increases \cite{Chen, Mahmassani1}.

The unimodality of distribution of route travel times under fixed Origin-Destination demand levels, however, may no longer hold when there is a group player which purposefully coordinates the routing of a large fleet of vehicles and deliberately induces unbalanced day-to-day distribution of demand on different routes. 
In \cite{Jamroz2024} different behavioural strategies maximizing immediate reward of a fleet controller were discussed and it was hypothesized, based on numerical experiments, that 'malicious' behavioural strategy, consisting in making travel times as large as possible for independent human drivers, might be the one preferred by market-share maximizing fleets of vehicles. Coincidentally, the 'malicious' strategy resulted in large oscillations of the number of drivers and travel times on different routes. In \cite{JamrozIFA} it was shown that, in simple scenarios, Stackelberg equilibria (see e.g. \cite{CorreaStackelberg, StackelbergRoughgarden}) for malicious fleet controller necessarily involve highly-variable mixed strategies made up of two different routing patterns. In this paper, we abandon explicit one-day behavioural objectives and focus on market share maximization in a market with diverse human attitudes towards adoption of CAVs. We discover that in such conditions the high variability of efficient market strategies emerges naturally. More precisely, we study equilibrium scenarios with a fleet of CAVs which explicitly uses mixed (in the game-theoretical sense) Stackelberg routing strategies, where two or more significantly different routing patterns are used with non-zero probabilities, and compare them to equilibrium scenarios where only deterministic Stackelberg routing is allowed. We discover that in scenarios with heterogenous attitudes towards adopotion of CAVs, market maximization can \emph{only} be achieved by using mixed strategies which may result in day-to-day travel time distributions which are bimodal with one peak near free-flow conditions, and one peak in the highly congested regime, Fig. \ref{Fig_traveltime}. 
\begin{figure}
\centering
\label{Fig_traveltime}
\includegraphics[scale=0.5]{"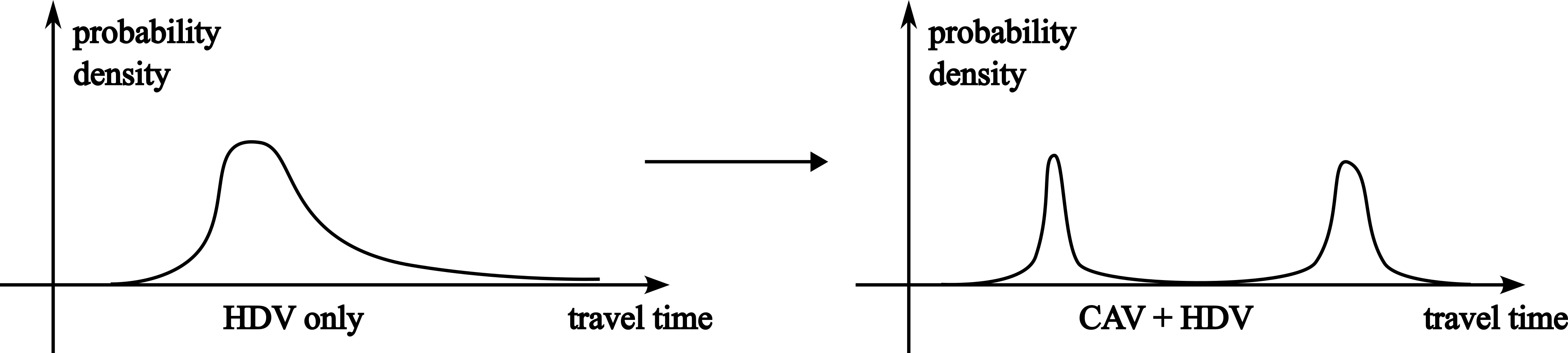"}
\caption{Day-to-day travel time distribution on a given route in typical urban settings (left) and likely travel time distribution when a fleet of CAVs uses probabilistic mixed routing. Note that even if the mean is the same, the $95$th percentile is shifted towards larger travel times which would inconvenience independent drivers who would need to depart significantly earlier in order to arrive on time on most days.}
\end{figure}
Such day-to-day travel time distributions, unknown in the previous transportation research, will result in unreliable networks and high buffer indices (schedule times shifts), seriously degrading the quality of urban traffic. To reach these conclusions, we develop a rigorous mathemical framework on the one hand and present carefully designed, illustrative examples, on the other.

The paper is organised as follows. In Section \ref{Sec_model} we introduce the model and the related equilibria, formulate our research questions and discuss HDV - CAV switching dynamics. 
In Section \ref{Sec_OfferProfiles} we introduce a general framework for considering the mean travel times, a fleet of CAVs can offer individually to drivers, which we call \emph{offer profiles}. In Section \ref{Sec_greedy} we introduce and study theoretically a greedy algorithm which verifies the  feasibility of travel time offer profiles, i.e. whether they can be realized in actual day-to-day asssignment of drivers to routes. Section \ref{Sec_DiscountFactorProfiles} introduces discount factor profiles which quantify heterogenous attitudes towards adoption of CAVs and relates them to optimal offer profiles. Section \ref{Sec_illustrative} presents four scenarios, which illustrate the theoretical framework introduced in previous sections and highlight the disturbing emergent phenomena. In Section \ref{Sec_Discussion} we summarize the results and discuss further research directions and open problems. Finally, in Appendix we provide more general utility-based models and discuss potential practical consequences of mixed routing. We also present results discussing routing every driver via two routes only omitted from the main text. 
\begin{figure}
\centering
\includegraphics[scale=0.36]{"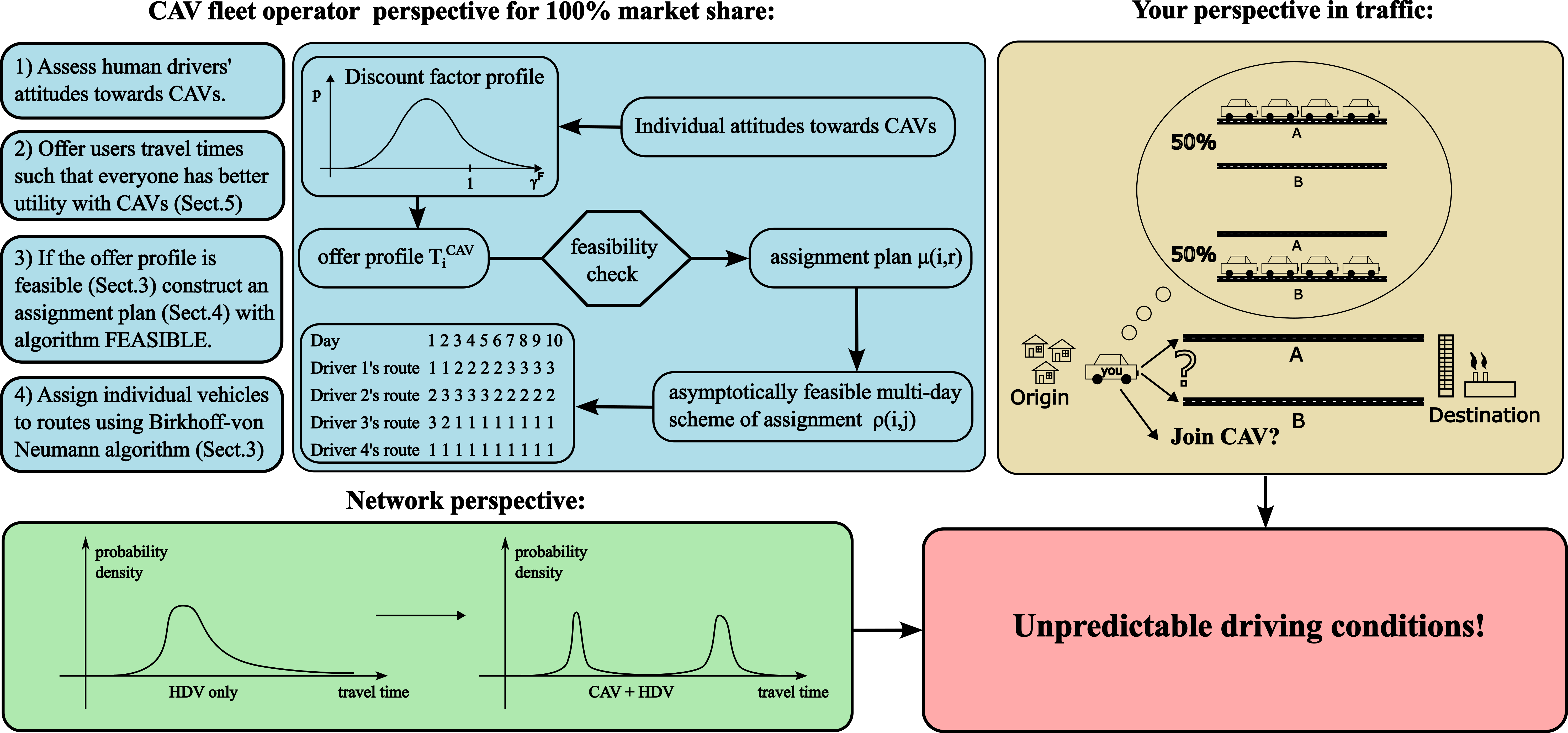"}
\label{Fig_abstr}
\caption{Results of the paper at a glance. CAV fleet operator uses a multi-stage pipeline to decide which routing can maximize the market share. In most realistic cases this results in you as a road user facing the choice: either use an independent HDV and select from routes with enormous uncertainty, or join the fleet and obtain better expected utility. From the network point of view, the fleet of CAVs can generate congestion on routes in a stochastic way resulting in travel times which have two separated peaks. The human perspective and network perspective jointly make up the unpredictability of driving conditions and put human drivers at a disadvantage vs. CAVs, which generate these conditions deliberately to maximize market share.}
\end{figure}

\subsection{Background}

\noindent Interaction between autonomous vehicles and human drivers is a prolific research area that has many interesting aspects such as microscopic driving behaviour \cite{Farah2022, Gora2020}. The interaction between human drivers and coordinated fleets of autonomous vehicles in future transportation systems, however, seems to be understudied. The research related to adoption of AVs, namely, has so far focused mostly on individual factors determining CAV uptake and discussing static levels of adoption in relation to automation level, e.g. \cite{Ardeshiri, Correia, Harrison}, or different routing outcomes \cite{Wang, Jamroz2024}. The dynamic market evolution in systems with mature CAV as a service offer where drivers are free to switch between driving/routing independently and subscribing to a coordinated fleet of autonomously driving and routing vehicles remain unexplored. 

Importantly, the dynamics of human-only (no CAVs) systems can already exhibit complex long run phenomena such as unstable or multiple attractors with research in this direction ranging from consideration of simple differential dynamical systems  to full-scale agent-based or aggregated stochastic processes, e.g. \cite{Smith, Horowitz, SmithWatling, CaCa, SmithEtAl}. Adding autonomous vehicles into such a system adds a further layer of complexity: the autonomous agents may behave selfishly, cooperatively or semi-cooperatively; they may interact differently with human vehicles than they do with other autonomous vehicles; and they may be ‘tuned’ to influence the system evolution in different ways, whether for the benefit or not of the different groups of agents.

Studying complex interactions such as HDV-CAV dynamics, one usually accounts for certain aspects while aggregating the others. In \cite{Bitar2022}, e.g., the authors investigated microscopic driving conflict interactions and made long-term predictions based on the evolution of the system using (evolutionary) game theory. Nevertheless, a comprehensive theory of adoption of CAVs allowing for forecasting long term properties in mature markets with CAVs treated as equal road-users has been lacking. 

In this contribution we initiate research to fill this gap, proposing behavioral models of CAV uptake derived from factors such as perceived efficiency of routing and study their properties, especially equilibria, stable states and long-term convergence, leaving the more dynamical aspects to future research.

\section{Model formulation and research questions}
\label{Sec_model}

We study day-to-day dynamics in a system with one Origin-Destination pair and $R$ parallel routes, Fig. \ref{Fig_OD} as an abstraction of urban traffic system, as well as drivers, indexed by $i$, who every day select/are assigned a route along which they travel. 
Each driver every day chooses between two modes -- HDV and CAV -- based on disutilities (perceived costs) $u_i^{HDV}, u_i^{CAV}$ of each mode. Having chosen HDV, they furthermore choose a route $r \in \{1,\dots,R\}$, along which they travel. Having chosen CAV, they are individually assigned a route by the fleet controller. 
\begin{figure}[h!]
\centering
\includegraphics[scale=0.8]{"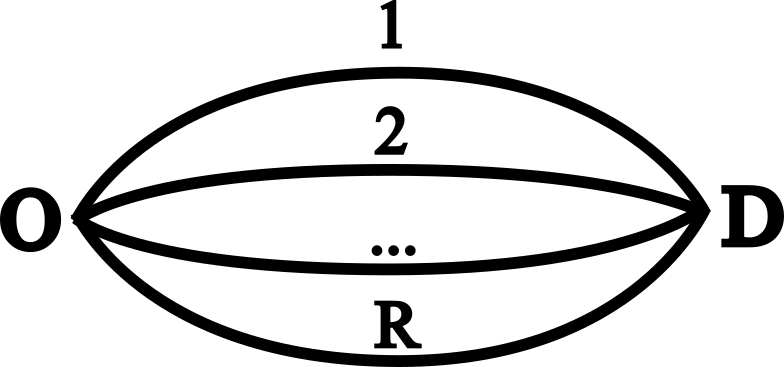"}
\caption{The considered system consists of $R$ independent parallel routes between Origin (O) and Destination(D).}
\label{Fig_OD}
\end{figure}
The travel times along different alternatives are given by macroscopic functions $t_1(\bold{q}), \dots, t_R(\bold{q})$, e.g. BPR, where $\bold{q} = (q_1,\dots,q_R)$ is the flow  vector accounting for the total flow of vehicles via different routes. The {\bf state of the system} on a given day $J$ is thus characterized by:
\begin{itemize}
\item $Mode(i) \in \{HDV, CAV\}$, i.e. every driver's chosen mode -- HDV or CAV, 
\item $r(i) \in \{1,2,\dots, R\}$ -- the route chosen by driver $i$ if $Mode(i) = HDV$ or the route assigned to driver $i$ if $Mode(i) = CAV$,
\item $u_i^{HDV}, u_i^{CAV}$ -- the disutilities (perceived costs) of using $HDV$ or $CAV$, respectively, for every $i$,
\item history of travel times $t_r^j$ on days $j = 1,2,\dots, J$ via routes $r = 1,2,\dots, R$,
\item routing schemes and mean travel times offered by the fleet operator to every driver. 
\end{itemize}
\subsection{Dynamics of the system}
After a given day of travel $J$, every driver updates the disutility $u_i^{HDV}$ using the most recent travel times $t_r^J$ and updates the disutility $u_i^{CAV}$ based on the, updated by the fleet operator, mean travel time and routing pattern offered by the fleet of CAVs. The next day $J+1$, every driver  reconsiders the mode choice and chooses (if mode is HDV) or is assigned (if mode is CAV) one of the routes $1,\dots, R$, see Fig. \ref{Fig_choice}. More elaborate models including e.g. departure time choice and mode specific constants are discussed is Appendix \ref{Sec_General}.

\begin{figure}[h!]
\centering
\includegraphics[scale=0.6]{"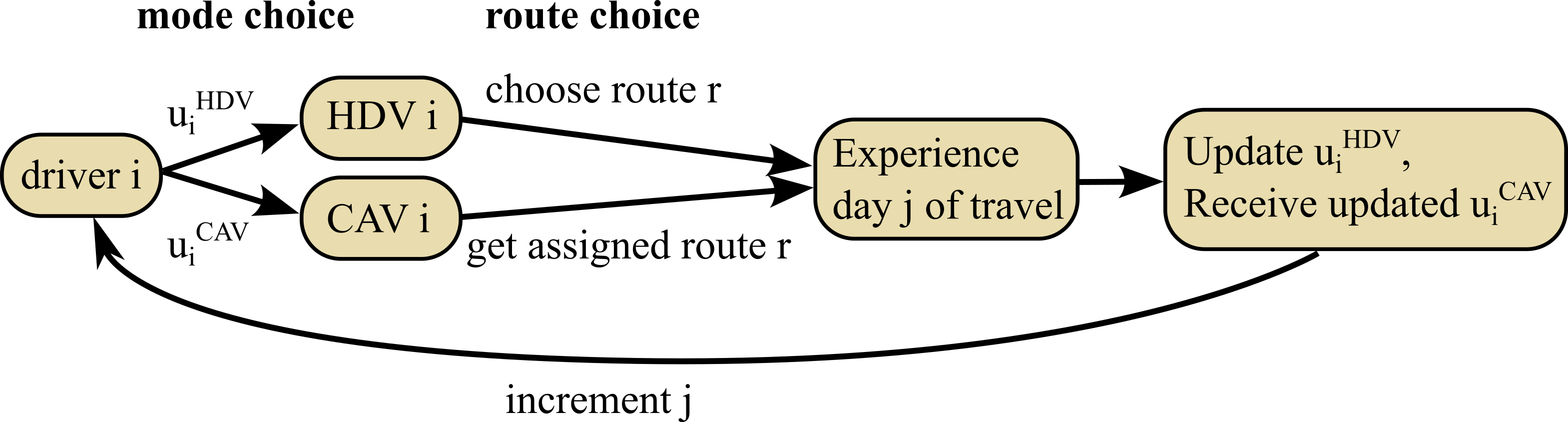"}
\caption{Dynamics of the considered system. Every driver $i$ makes a decision whether to use an HDV or CAV based on comparing the disutilities (perceived costs) of the two options, $u_i^{HDV}$ and $u_i^{CAV}$. If $u_i^{HDV}<u_i^{CAV}$ then they choose HDV. If $u_i^{HDV}>u_i^{CAV}$ then they choose CAV. Otherwise, they stick to the mode from the previous day. Then, on a given day of travel $j$, if HDV is the chosen mode, the driver selects one of the routes $1,\dots, R$ based typically on minimization of expected travel time and drives from Origin to Destination along the selected route. Contrariwise, if CAV is the chosen mode, then driver $i$ is assigned a route $r$ and the autonomous vehicle takes driver $i$ from Origin to Destination. After a finished day of travel, the drivers update their utilities of using an HDV (based typically on which route was the fastest on the past days). The fleet informs driver $i$ of the long-term utility it can offer to driver $i$. Periodically (every day or every several days), driver $i$ reconsiders whether to change the mode. Whether driver $i$ swaps the mode or not, they face the route choice on the following day $j+1$ and this decision - feedback process is looped indefinitely. At dynamic equilibrium the mode choice is fixed for every driver $i$ and no driver is inclined to change it. At nested equilibrium, the mode choices, the route choices of drivers using HDV and routing patterns for every $CAV$ are fixed and not only no driver is inclined to change the mode but also no HDV user is inclined to choose a different route. }
\label{Fig_choice}
\end{figure}
\subsection{User behaviour}
\noindent The {driver behaviour} is characterized by how the disutilities $u_i^{HDV}$ and $u_i^{CAV}$ are set. In this paper, we usually assume that $u_i^{HDV}$ corresponds to the expected travel time  on the fastest route, $t_{min}$, whereas $u_i^{CAV}$ is based on the routing pattern offered by the fleet operator and typically may be proportional to mean long-term offered travel time. Then HDVs $i$ are typically assumed to select the fastest route available, whereas CAVs $i$ are assigned routes by the fleet controller such that, in the long run, the average offered travel times correspond to actually experienced travel times for every fleet member. In this setting, the {\bf control variable} we study is the {\bf individualized routing pattern} of the fleet controller yielding {average travel times offered to drivers $i$}. 

\subsection{Research questions}
\label{Sec_ResQ}
The questions we address in this paper are: 
\begin{itemize}
\item[{\bf (Q1)}] What mean travel time $T_i^{CAV}$ should the fleet controller offer to driver $i$ such that $u_i^{CAV} < u_i^{HDV}$ for as many drivers $i$ as possible (market-share maximization)?
\item[{\bf (Q2)}] Can mean travel times $T_i^{CAV}$ be realized in multi-day traffic for all $i$ simultaneously, i.e. is a given offer profile feasible? 
\item[{\bf (Q3)}] What network conditions result from a given (especially market-share maximizing) feasible travel time offer profile? 
\end{itemize}

\noindent We study {\bf (Q2)} in Sections \ref{Sec_OfferProfiles}-\ref{Sec_greedy}, where we introduce a general framework of individualized offer profiles. In particular, we discuss the conditions under which individual assignment plans yielding travel times $T_i^{CAV}$ exist. In the case of two available routes, these conditions have a simple analytical form, while in the case of many routes, we introduce an algorithm which verifies the feasibility of offer profiles and constructs the corresponding assignment plans. The answer to Question {\bf (Q1)} depends on the attitudes individual drivers have towards collectively routing CAVs and is discussed in Section \ref{Sec_DiscountFactorProfiles}. Illustrative examples accross four scenarios are presented in Section \ref{Sec_illustrative}, answering {\bf (Q3)}. As mentioned before, we generally consider human drivers' (HDV) choice as consisting in selecting route only. Nevertheless, to put our results in a broader context, in Appendix \ref{Sec_General} we discuss also more general models when departure times of time-critical trips can be adjusted in response to changing driving conditions and observe that dramatic schedule shifts are to be expected when CAVs strive to maximize market share.

\subsection{Equilibria}
In this paper we mostly focus on the $100\%$ CAV market share in {\bf Fleet-Human Equilibrium}, see Definition \ref{def_FHE} below, 
where $u_i^{CAV}\le u_i^{HDV}$ for every $i$ and CAV fleet routing patterns are fixed.

We note that in this setting the temporal dynamics of CAV vs. HDV switching is irrelevant, compare Section \ref{Sec_switching}.

\begin{definition}(Fleet-human equilibria)
\label{def_FHE}
A mixed traffic system consisting of one fleet of CAVs and independent human drivers is in 
\begin{itemize}
\item {\bf Dynamic Fleet-Human equilibrium (DFHE)} if no driver has an incentive to switch from HDV to CAV or vice versa.
\item {\bf Nested Fleet-Human Stackelberg equilibrium (NFHE)} if no driver has an incentive to switch from HDV to CAV or vice versa and the system composed of atomic individual (or infinitesimal) human driver players and a group fleet player is in Stackelberg equilibrium at the level of route choice. 
\end{itemize}
\end{definition}

\begin{remark}
The realization of Fleet-Human equilibria depends on the behavioural objectives the players optimize, see Example \ref{ex_fhe} below. 
\end{remark}

\begin{example}[Examples of Fleet-Human equilibria]
\label{ex_fhe}
\begin{enumerate}
\item[i)] HDV-only system with no option to switch to CAV is always in Dynamic FHE. If, additionally, it is in User Equilibrium at the level of route choice then the system is in Nested FHE. 
\item[ii)] A system with $100 \%$ of CAVs where no driver has an incentive to switch is in NFHE if the fleet's sole objective is the maximization of market share.
\item[iii)] If the objective of fleet is bi-level: first maximize market share and then minimize cost then a system with $100 \%$ of CAVs is in NFHE only if it is optimized cost-wise, i.e. the fleet uses the most cost-efficient routing which guarantees $100\%$ market share. In particular the routing considered in Section \ref{ex_3} is not in NFHE as it is not optimized in terms of cost to the fleet. 
\item[iv)] If the market share of CAVs is less than $100\%$ then the system can be in Stackelberg NFHE only if there exists no Stackelberg routing which achieves a higher market share. 
\item[v)] A system consisting of both CAVs and HDVs, which is disequlibrated at the level of route choice (e.g. when HDVs take time to approach equilibrium) can be in DFHE as long as no driver has (and never will have) an incentive to switch mode. 
\end{enumerate}
\end{example}

\begin{openproblem} The structure of fleet-human equilibria is relatively straightforward in the scenarios with $100\%$ market share of CAVs considered in this paper. The landscape and structure of FHE in the general case is open.
\end{openproblem}

\subsection{Switching HDV $\leftrightarrow$ CAV models} 
\label{Sec_switching}
In this paper, with the exception of scenario studied in Section \ref{ex_4}, we consider equilibrium solutions, where every driver is committed to using HDV or CAV and hence the sign of $u_i^{CAV} - u_i^{HDV}$ is fixed for every $i$ and no driver has an incentive to switch. In this section we briefly hint at possible models of the dynamics of switching between HDV and CAV, urging the transportation research community to conduct behavioural studies which would clarify which model best reflects reality. Let us first discuss {\bf what it means to be a member of the fleet of CAVs}. There are two main variants:
\begin{enumerate}
\item[i)] To become part of the fleet of CAVs one has to buy a car. Then switching $HDV \to CAV$ is potentially difficult and is a major decision. 
\item[ii)] All cars possess the automated driving mode. Using an HDV means switching off the autopilot. Using a CAV means switching on the AV mode and connecting to the fleet operator which provides, every day, routing instructions. 
\end{enumerate}
In this paper, we generally assume that the market is mature and switching between HDV and CAV fleet literally means 'flicking a switch' in your car. Clearly, there exist many variants such as 'subscription period', where the fleet can be joined only for e.g. a minimum period of half a year etc. We leave studying these scenarios to further research. We also assume that AV mode {\bf always} comes with collective routing. This does not have to be the case and other scenarios may involve drivers using AV mode without collective routing. These scenarios are also left for further research. 

To conclude, our {\bf basic model of switching $HDV \leftrightarrow CAV$ in a mature market} reads as follows: Every day, every driver computes (themselves, by gut feeling, or using a dedicated app) the disutility $u_i^{HDV}$ and  computes or gets computed the disutility $u_i^{CAV}$. Then:
\begin{itemize}
\item If $u_i^{HDV}<u_i^{CAV}$, driver $i$ chooses HDV for the next day.
\item If $u_i^{HDV}>u_i^{CAV}$, driver $i$ chooses CAV for the next day.
\item If $u_i^{HDV}=u_i^{CAV}$, driver $i$ sticks to the choice of mode from the previous day.
\end{itemize}

More complex models, such as modeling a nascent market using models like Bass model \cite{Bass} or models using more complex rule-based criteria/inertia/day-to-day variation of preferences based on market situation/social influence etc. are left open.

\section{Offer profiles and assignment plans}
Before we discuss which strategies are optimal for market maximization, we need to understand which of them are feasible, i.e. can be realized by the fleet operator as actual assignments of vehicles to routes on a sequence of days. Accordingly, in this section we present a general framework which enables rigorous study of individualized  travel time offers under fixed network conditions as well as fixed flows of CAVs and HDVs on every route. The drivers $i \in I$ may be considered infinitesimal (then $I$ is assumed to be the unit inverval $[0,1]$ with the Lebesgue measure $di$) or atomic (then $I$ can be a subset of integers, $I=\{1,2,\dots,q^{CAV}\}$ with the measure $di$ understood as counting measure, i.e. $\int_I f(i) di = \sum_{i=1}^{q^{CAV}} f(i)$ for any function $f$). In both settings (discrete and continuous), we first define the offer profiles, which amount to offering every driver $i \in I$ 
a positive real number $T_i^{CAV}$, which is the average travel time that is supposed to be realizable long-term (after many days of driving), and assignment plans, denoted $\mu(i,r)$, which prescribe proportions with which drivers $i$ should be routed via different alternatives $r$ in order to realize a given offer profile. 

\label{Sec_OfferProfiles}
\begin{definition}[{CAV travel time offer profile, assignment plan}]
Let $\bold{q^{CAV}} = (q^{CAV}_1, \dots, q^{CAV}_R)$ be a fixed assignment of CAV flow to routes (henceforth usually called briefly routing) and $q^{CAV} = \sum_{r=1}^R q_r^{CAV}$ be the total flow of CAVs. Let $\bold{t^{CAV}}= (t_1^{CAV},\dots,t_R^{CAV}) = (t_1(\bold{q^{CAV}}+\bold{q^{HDV}}), \dots, t_R(\bold{q^{CAV}}+\bold{q^{HDV}}))$ be the corresponding travel time vector, which for fixed equilibrated $\bold{q^{HDV}}$ is also fixed.
\begin{enumerate}
\item[i)] 
 A CAV travel time {\bf offer profile} subject to routing $\bold{q^{CAV}}$ and route travel times $\bold{t^{CAV}}$ is any measurable function $T^{CAV}: I \to [\min_r t^{CAV}_r, \max_r t^{CAV}_r]$ satisfying the compatibility condition
\begin{equation}
\label{eq_compatibility}
\frac 1 {|I|} \int_I T^{CAV}_i di = \frac 1 {q^{CAV}} \bold{q^{CAV}}\cdot \bold{t^{CAV}} =: \overline{t^{CAV}}
\end{equation}
i.e. the average offered travel time is equal to the average travel time actually experienced in the system, where $\bold{q^{CAV}} \cdot \bold{t^{CAV}} = \sum_{r=1}^R q^{CAV}_r t^{CAV}_r$ is the dot product. 
\item[ii)] An {\bf assignment plan} $\mu$ inducing the offer profile $T^{CAV}$ is a measurable function $\mu: I \times \{1,\dots, R\} \to [0,1]$, $|I| = q^{CAV}$, satisfying:
\begin{eqnarray}
\mbox{(correct route flows)}&\quad \int_I \mu(i, r)di = q_r^{CAV}& \mbox{ for every } r,\label{APmarg1}\\
\mbox{(proportions add up to 1)}&\quad \sum_r \mu(i, r) = 1& \mbox{ for a.e. } i,\label{APmarg2}\\
\mbox{(correct average travel time)}&\quad \sum_r t_r^{CAV} \mu(i,r) = T_i^{CAV}& \mbox{ for a.e. } i. \label{APtime}
\end{eqnarray}

\item[iii)] A {\bf feasible} CAV travel time {\bf offer profile} is a CAV travel time offer profile which is induced by at least one assignment plan.
\end{enumerate}
\end{definition}

\begin{remark}
\label{rem_induced}
\begin{enumerate}
\item[i)]  
$(I, di)$ can be treated as a probability space,  with total mass of $I$ not necessary equal $1$, and $T^{CAV}$ as a random variable. The corresponding induced distribution $\tau$ on the range of $T^{CAV}$ is given by $\tau(B) = \int_{\{i: T_i^{CAV} \in B\}} di$ for Borel sets $B\subset [\min_r t^{CAV}_r, \max_r t^{CAV}_r]$. We will write $T^{CAV} \sim \tau$ and say "$T^{CAV}$ is $\tau$-distributed". Moreover, given the distribution $\tau$ and the underlying space $(I,di)$ one can reconstruct $T^{CAV}$.  
\item[ii)] Similarly, function $i \mapsto (\mu(i,1),\dots,\mu(i,R)) \in \mathbb{R}^R$ induces a distribution $\nu$ on the unit simplex $\Delta^{R-1} = \{(q_1,\dots,q_R): q_1+\dots+q_R = 1, q_1,\dots,q_R\ge 0\}$ by $\nu(B) = \int_{\{i: (\mu(i,1), \dots, \mu(i,R)) \in B\}} di$ for Borel subsets $B \subset \Delta^{R-1}$. We write $\mu \sim \nu$. Furthermore, given $\nu, I$ one can reconstruct $\mu$. 
\item[iii)] Given $\tau$ and $\nu$ we will say that $\nu$ generates $\tau$ if $\tau(B) = \nu(\{\bold{q}: \bold{q}\cdot \bold{t^{CAV}}\in B\})$ or, informally, $\tau(t) "=" \int_{\{\bold{q}:\bold{q}\cdot \bold{t^{CAV}}=t\}} d\nu$.
\end{enumerate}

\end{remark}

\begin{example}
\label{Ex_sym}
If $\bold{q^{CAV}}$ corresponds to system optimum then
$T_i^{SO, sym} = \overline{t^{CAV}}$ for every $i$ is a {\bf symmetric system optimal} offer profile. The corresponding symmetric (the same for every driver) assignment plan is given by $\mu(i,r) =  {q^{CAV}_r}/{q^{CAV}}$, which corresponds to routing every driver proportionally to the flow on every route. E.g. if $\bold{q^{CAV}} = (3,5,6)$, i.e. the flows on three routes are $3, 5, 6$, then every driver is routed via the first route on $3/14$ days, via the second on $5/14$ days and via the third on $6/14$ days. The corresponding measure $\nu$ is $\nu = \delta_{(q_1^{CAV}/q^{CAV}, \dots, q_r^{CAV}/q^{CAV})}$, where $\delta$ is the Dirac mass, which in our example is $\nu = \delta_{(3/14, 5/14, 6/14)}$, compare \cite{Hoffmann}.  
\end{example}

\begin{proposition}[Basic properties of offer profiles and assignment plans]
\label{Prop_AssPlans_prop}
\begin{enumerate}
\item[i)] Every assignment plan $\mu$ defines an offer profile $T^{CAV}$.
\item[ii)] If there are two routes only, then every offer profile $T^{CAV}$ is feasible. Moreover, if the travel times via the two routes are different, the associated assignment plan $\mu$ is unique (almost everywhere).
\item[iii)] If the number of routes is $\ge 3$ then not every offer profile may be feasible
\end{enumerate}
\end{proposition}
\begin{proof}
\begin{enumerate}
\item[i)] Integrating \eqref{APtime}, via which we define $T^{CAV}$, over $I$ and using \eqref{APmarg1} we obtain $\int_I T_i^{CAV} di = \int_I \sum_r t^{CAV}_r\mu(i,r)di = \sum_r t_r^{CAV} q_r^{CAV}$. Dividing by $q^{CAV}$ we obtain \eqref{eq_compatibility} and conclude.
\item[ii)] The case of equal travel times is trivial as every driver can only be offered the same travel time, however there is clearly no uniqueness of assignment plan. Now, assume $t_1^{CAV} < t_2^{CAV}$. For every $i$ there exist unique $\mu(i,1), \mu(i,2)$, given by $\mu(i,1) = \frac {t_2^{CAV} - T_i^{CAV}}{t_2^{CAV} - t_1^{CAV}}$ and $\mu(i,2) = \frac { T_i^{CAV}-t_1^{CAV}}{t_2^{CAV} - t_1^{CAV}}$, such that \eqref{APmarg2}-\eqref{APtime} are satisfied:
\begin{eqnarray}
\mu(i,1)t_1^{CAV} + \mu(i,2)t_2^{CAV} &=& T_i^{CAV} \label{eq_mutmut}\\
\mu(i,1)+\mu(i,2) &=& 1, \label{eq_mumu}
\end{eqnarray}
which proves uniqueness. To verify \eqref{APmarg1}, we note that by integrating \eqref{eq_mutmut} we have $$\left(\int_I\mu(i,1)di\right) t_1^{CAV} + \left(\int_I\mu(i,2)di\right) t_2^{CAV} = \int_I T_i^{CAV} di$$ whereas integration of \eqref{eq_mumu} yields $$\left(\int_I\mu(i,1)di\right) + \left(\int_I\mu(i,2)di\right) = q^{CAV}.$$ On the other hand, by \eqref{eq_compatibility} $$q_1^{CAV} t_1^{CAV} + q_2^{CAV}t_2^{CAV} = \int_I T_i^{CAV} di$$ and $$q_1^{CAV} + q_2^{CAV} = q^{CAV}.$$
As $t_1^{CAV} \neq t_2^{CAV}$ the coefficients of the linear combination are unique and hence $\int_I \mu(i,r) = q_r^{CAV}$ for $r=1,2$. Therefore, \eqref{APmarg1}-\eqref{APtime} hold and so $\mu$ is an assignment plan inducing $T^{CAV}$. 

\item[iii)] Consider $\bold{q^{CAV}} = (0.25,0.5,0.25)$ with $\bold{t^{CAV}} = (10,20,30)$. Then offer profile ${T^{CAV}} \sim 0.5 \delta_{10} + 0.5 \delta_{30}$, offering to half of the drivers travel time $10$ and to the other half travel time $30$, is not feasible even though $\overline{T^{CAV}} = 20 = \overline{t^{CAV}}$. Contrariwise, $T^{CAV} \sim \delta_{20}$ is feasible however the corresponding assignment plan is non-unique, e.g. $\mu_1 \sim \delta_{(0.25,0.5,0.25)}$ (symmetric) and $\mu_2 \sim 0.5\delta_{(0,1,0)} + 0.5\delta_{(0.5, 0, 0.5)}$ both induce $T^{CAV}$. The former assignment plan assigns every driver to route $1$ on a quarter of days, route $2$ on half of days and route $3$ on a quarter of days. The latter assignment plan assigns half of the drivers always to route $2$ and the other half to route $1$ (on half of days) and route $3$ (on half of days). In both assignment plans, every driver experiences on average travel time $20$. 
\end{enumerate}
\end{proof}
In Appendix \ref{Sec_proofs} we show that assignment plans where drivers are routed via more than $2$ alternatives can be decomposed into plans where each driver uses no more than $2$ routes. 

A feasible offer profile prescribes proportions with which every driver should be routed along each alternative. However, it is not clear a priori whether there exists actual assignment of drivers to routes that results in the given proportions and, consequently, the desired average travel times. Below, we show that it is indeed so. We begin by defining multi-day assignment schemes, which, every day, assign individual vehicles to routes, and asymptotically feasible schemes which, in the long run, approach the correct average travel times. 

\begin{definition}[Multi-day assignment scheme of individual drivers to routes]
By multi-day assignment of individual drivers to routes we understand a measurable function $\rho: I \times \{1,2,\dots\} \to \{1,2, \dots, R\}$ which maps the couple $(i, j)$ to the route driver $i$ is supposed to take on day $j$. 
\end{definition}
\begin{definition}
A multi-day scheme of assignment of individual drivers to routes $\rho$ is called \emph{asymptotically feasible} if $\frac 1 J \sum_{j=1}^J t_{\rho(i, j)} \to T^{CAV}$ for every $i$ as $J \to \infty$, i.e. the long term average travel time approaches the travel time prescribed by the offer profile $T^{CAV}$. 
\end{definition}

\begin{figure}
\centering
\includegraphics[scale=0.46]{"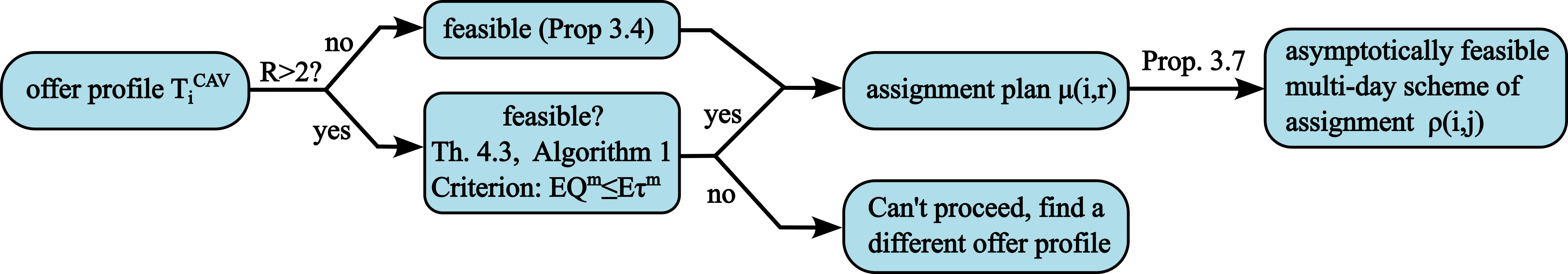"}
\label{fig_results}
\caption{Summary of the technical results of the paper. An offer profile offering every driver $i$ a given mean travel time $T_i^{CAV}$ is always feasible for two routes (Proposition \ref{Prop_AssPlans_prop}), i.e. it admits an assigment plan $\mu(i,r)$ such that $\mu(i,r)$ is the proportion with which driver $i$ is routed via alternative $r$ and, given these proportions, the mean travel time of driver $i$ is $T_i^{CAV}$. For three or more routes, Algorithm \ref{Alg_feasible} and Theorem \ref{Th_feasibility} provide a useful criterion for verifying feasibility as well as a procedure of construction of assignment plans. Given assignment plan $\mu$, there always exists and can be constructed (Proposition \ref{Prop_asymptotically}) a multi-day scheme of assignment which assigns route $\rho(i,j)$ to CAV user $i$ on day $j$ such that the proportions given by assignment plan $\mu(i,\cdot)$ are fulfilled long-term for every user $i$.}
\end{figure}

\begin{proposition}[Feasible travel time offer can be realized as multi-day assignment scheme]
\label{Prop_asymptotically}
Every feasible CAV travel time offer profile admits an asymptotically feasible scheme of assignment of individual drivers to routes such that the flows on routes are every day equal to given fleet flows. 
\end{proposition}
\begin{proof}
We only discuss the discrete case, as it has a clear real-world interpretation, the continuous case addressable by approximation. Namely, for every $i \in I=\{1,\dots,q\}$ the assignment plan $\mu(i,r)$ prescribes the proportion with which driver $i$ should be routed via $r$. The symmetric considerably simpler case of $\mu(i,r)$ independent of $i$ (the same proportions for every driver) was discussed in detail in \cite{Hoffmann}. 
Let $\bold{q^{CAV}} = (q^{CAV}_1,\dots,q^{CAV}_R)$ where $q^{CAV}_1, \dots, q^{CAV}_R$ are nonnegative integers. Let  $T^{CAV}$ be a feasible travel time offer profile and fix the assignment plan $\mu$ which induces it. 
Then the matrix $M$ obtained by replacing the $r$-th column $\mu(\cdot,r)$ of $\mu$ by $q_r^{CAV}$ scaled copies of this column:
\begin{equation*}
M = \text{\Huge[}\underbrace{\frac{\mu(\cdot,1)}{q_1^{CAV}},\dots, \frac{\mu(\cdot,1)}{q^{CAV}_1}}_{q^{CAV}_1}, \underbrace{\frac{\mu(\cdot,2)}{q^{CAV}_2},\dots, \frac{\mu(\cdot,2)}{q^{CAV}_2}}_{q^{CAV}_2}, \dots, \underbrace{\frac{\mu(\cdot,R)}{q^{CAV}_R},\dots, \frac{\mu(\cdot,R)}{q^{CAV}_R}}_{q^{CAV}_R} \text{\Huge]}
\end{equation*}
 is a doubly stochastic $q \times q$ matrix. By Birkhoff-von Neumann theorem \cite{Birkhoff} there exist permutation matrices $P_1, \dots, P_Z$ (where $Z \le q^2$), such that $$M = \sum_{z=1}^Z \theta_z P_z$$ for some $\theta_1, \dots, \theta_z \ge 0$ with $\sum_{z=1}^Z \theta_z = 1$.  Now, each permutation matrix $P_z$ corresponds to a daily assignment. Indeed, by recombining the columns of matrix $P_z$ we obtain a $q \times r$ matrix $A_z$ whose first column is the sum of the first $q_1^{CAV}$ columns of $P_z$, the second column is the sum of the next $q_2^{CAV}$ columns of $P_z$, the third column is the sum of the next $q_3^{CAV}$ columns and so on. Matrix $A_z$ corresponds to assigning driver $i$ to the unique route $r$ such that $A_z(i,r) = 1$. The proof is finished by observing that there exists a sequence $z_1, z_2, z_3, \dots \subset \{1,\dots,Z\}^\infty$ such that the frequency of $z$ approaches 
$\theta_z$ for $z \in \{1,\dots,Z\}$, i.e. 
\begin{equation}
\label{eq_freq} 
\frac 1 J |\{j \in \{1,\dots, J\}: z_j = z\}| \xrightarrow{J \to \infty} \theta_z.
\end{equation}
The resulting multi-day scheme of assignment of individual drivers to routes given by $\rho(i,j) := \arg\max_r  A_{z_j}(i,r)$ is asymptotically feasible since $\mu = \sum_{z=1}^Z \theta_z A_z$ which implies
\begin{eqnarray*}
\frac 1 J \sum_{j=1}^J t^{CAV}_{\rho(i, j)} &=& \sum_z \frac 1 J |\{j \in \{1,\dots, J\}: z_j = z\}| t^{CAV}_{\arg\max_r  A_{z}(i,r)}  \xrightarrow{J \to \infty} \sum_z \theta_z t^{CAV}_{\arg\max_r  A_{z}(i,r)}\\ &=& \sum_z \theta_z \left(\sum_r A_z(i,r) t^{CAV}_r\right) = \sum_r  \left(\sum_z \theta_z A_z(i,r)\right) t^{CAV}_r
= \sum_r \mu(i,r) t_r^{CAV} = T_i^{CAV}.
\end{eqnarray*}

\begin{remark}
The explicit construction of sequence $z_j$ in the proof of Proposition \ref{Prop_asymptotically} is not necessary, as it suffices to apply a fully probabilistic assignment $\rho(i,j):=\arg\max_r  A_{z}(i,r)$ with probability $\theta_z$ for $z=1,\dots,Z$ and every $j$ independently, which by the law of large numbers converges almost surely to the correct proportions. Nevertheless, explicit sequence $z_j$ allows for better potential control of the properties of the sequence of assignments, such as convergence rate. 
\end{remark}

\end{proof}

\begin{example}
\label{ex_10days}
Let ${\bf q^{CAV}} = (2,1,1)$, $q=4$. For 
$\mu = \begin{bmatrix}
0.2 & 0.3 & 0.5\\
0.0 & 0.6 & 0.4\\
0.8 & 0.1 & 0.1\\
1.0 & 0.0 & 0.0
\end{bmatrix}
$ we have $M = \begin{bmatrix}
0.1 & 0.1 & 0.3 & 0.5\\
0.0 & 0.0 & 0.6 & 0.4\\
0.4 & 0.4 & 0.1 & 0.1\\
0.5 & 0.5 & 0.0 & 0.0
\end{bmatrix}
$ which is doubly stochastic. The corresponding (non-unique) Birkhoff decomposition is e.g.
\begin{eqnarray*}
10 M = \begin{bmatrix}
1 & 1 & 3 & 5\\
0 & 0 & 6 & 4\\
4 & 4 & 1 & 1\\
5 & 5 & 0 & 0
\end{bmatrix}
= \begin{bmatrix}
1 & 0 & 0 & 0\\
0 & 0 & 1 & 0\\
0 & 0 & 0 & 1\\
0 & 1 & 0 & 0
\end{bmatrix}
+ \begin{bmatrix}
0 & 1 & 0 & 0\\
0 & 0 & 0 & 1\\
0 & 0 & 1 & 0\\
1 & 0 & 0 & 0
\end{bmatrix} +
\begin{bmatrix}
0 & 0 & 3 & 0\\
0 & 0 & 0 & 3\\
3 & 0 & 0 & 0\\
0 & 3 & 0 & 0 
\end{bmatrix} + 
\begin{bmatrix}
0 & 0 & 0 & 4\\
0 & 0 & 4 & 0\\
0 & 4 & 0 & 0\\
4 & 0 & 0 & 0 
\end{bmatrix} + 
\begin{bmatrix}
0 & 0 & 0 & 1\\
0 & 0 & 1 & 0\\
1 & 0 & 0 & 0\\
0 & 1 & 0 & 0 
\end{bmatrix} 
\end{eqnarray*}
which yields $M = 0.1 P_1 + 0.1 P_2 + 0.3 P_3 + 0.4 P_4 + 0.1 P_5$ and $\mu = 0.1 A_1 + 0.1 A_2 + 0.3 A_3 + 0.4 A_4 + 0.1 A_5$, where 
\begin{eqnarray*}
A_1 = \begin{bmatrix}
1 & 0 & 0\\
0 & 1 & 0\\
0 & 0 & 1\\
1 & 0 & 0
\end{bmatrix}, 
A_2 = \begin{bmatrix}
1 & 0 & 0\\
0 & 0 & 1\\
0 & 1 & 0\\
1 & 0 & 0
\end{bmatrix}, 
A_3 =
\begin{bmatrix}
0 & 1 & 0\\
0 & 0 & 1\\
1 & 0 & 0\\
1 & 0 & 0 
\end{bmatrix}, 
A_4 =  
\begin{bmatrix}
0 & 0 & 1\\
0 & 1 & 0\\
1 & 0 & 0\\
1 & 0 & 0 
\end{bmatrix},  
A_5 = \begin{bmatrix}
0 & 0 & 1\\
0 & 1 & 0\\
1 & 0 & 0\\
1 & 0 & 0 
\end{bmatrix} 
\end{eqnarray*}
We have $\theta_1 = 0.1, \theta_2 = 0.1, \theta_3 = 0.3, \theta_4 = 0.4, \theta_5 = 0.1$ and so the simplest sequence $z_1, z_2, \dots $ satisfying \eqref{eq_freq} is given by repeating $1,2,3,3,3,4,4,4,4,5$. This corresponds to drivers assigned on the first $10$ days (the pattern is repeated, on the following days) as shown in Table \ref{tab1}. 
\begin{table}[h!]
\centering
\begin{tabular}{c||c|c|c|c|c|c|c|c|c|c|}
Day & 1 & 2& 3& 4& 5& 6& 7& 8& 9& 10\\
\hline
\hline
Driver 1's route & $1$&$1$&$2$&$2$&$2$&$3$&$3$&$3$&$3$&$3$\\
\hline
Driver 2's route & $2$&$3$&$3$&$3$&$3$&$2$&$2$&$2$&$2$&$2$\\
\hline
Driver 3's route & $3$&$2$&$1$&$1$&$1$&$1$&$1$&$1$&$1$&$1$\\
\hline
Driver 4's route & $1$&$1$&$1$&$1$&$1$&$1$&$1$&$1$&$1$&$1$\\
\hline
\end{tabular}
\caption{Assignment of drivers to routes on the first $10$ days of travel such that Driver 1's proportions of using routes $1$,$2$,$3$ are $0.2$,$0.3$,$0.5$, respecitvely, Driver 2's proportions are $0$,$0.6$,$0.4$, Driver 3's are $0.8$,$0.1$,$0.1$ and Driver 4's are $1.0$,$0$,$0$. Simultaneously, the flows on routes $1,2,3$ are equal $2,1,1$ every day. Such an assignment can be constructed, thanks to Proposition \ref{Prop_asymptotically} for any stipulated assignment plan $\mu$ prescribing the proportions of using different routes by each driver which are compatible with given route flows.}
\label{tab1}
\end{table}

\noindent Note that on each day we have two drivers assigned to route  $1$ and $1$ driver assigned to routes $2$ and $3$ which was stipulated by $\bold{q^{CAV}}$. If an offer profile $T^{CAV}$ is induced by $\mu$ then the above assignment scheme is asymptotically feasible and can be used by a fleet operator to achieve the desired travel times for the drivers. 
\end{example}

\begin{remark}
\label{Rem_compHof}
Proposition \ref{Prop_asymptotically} generalizes the results from \cite{Hoffmann} to the case of non-symmetric $\mu$. Note that in \cite{Hoffmann} Proposition 2.11 the daily assignments were generated as $P^jA_1$ for a given initial assignment matrix $A_1$ and a single permutation matrix $P$, which easily yielded the correct (the same for every user) proportions of routing via different alternatives. In the case of non-symmetric $\mu$, however, obtaining the correct proportions is not straightforward any longer and so in Proposition \ref{Prop_asymptotically} we use the Birkhoff algorithm which relies on optimal matching theorems. We also note that in the non-symmetric setting there is in general no hope of obtaining a cycle like in \cite{Hoffmann} and only asymptotic sequences can be produced. The cyclical assignment in Example \ref{Ex_sym} was possible only thanks to the special form (which we will not have in realistic use cases) of the assignment plan containing finite decimal fractions. 

In spite of that, \cite{Hoffmann} also discussed approximation algorithms in place of cyclical assignments and we conjecture that a greedy approximation algorithm similar to the algorithm  in \cite{Hoffmann} Proposition 2.17 could also be designed.  
\end{remark}
\noindent In view of Remark \ref{Rem_compHof} we have the following
\begin{openproblem}
In \cite{Hoffmann} the assignment sequences for \emph{symmetric} $\mu$ were optimized (with respect to different behavioral measures) to a large degree. Finding optimal assignments in the general non-symmetric setting considered here is open.  
\end{openproblem}

\section{Greedy algorithm for verifying the feasibility of offer profile $T^{CAV}$}
\label{Sec_greedy}
By Proposition \ref{Prop_AssPlans_prop}, every offer profile $T^{CAV}$ is feasible if there are two routes, this however is no longer true if there are more routes. Accordingly, in this section we present a greedy algorithm which verifies whether a given offer profile is feasible and constructs a corresponding assignment plan as a byproduct. We call $FEASIBLE((q_1,q_2,\dots,q_N),(t_1,t_2,\dots,t_N),\tau)$, see Algorithm \ref{Alg_feasible}, the function which returns 
\begin{itemize}
\item $(TRUE, \nu)$ if an offer profile subject to $(\bold{q},\bold{t}) = ((q_1,\dots,q_N),(t_1,\dots,t_N))$ and distributed according to $\tau$ is feasible. The corresponding assignment plan may be then distributed according to $\nu$. 
\item $(FALSE, \nu)$ if the offer profile distributed according to $\tau$ is not feasible. 
\end{itemize}
For notational convenience, in Algorithm \ref{Alg_feasible} by $\nu^{N-1}_{+n}$ we denote a measure on the $N-1$-dimensional simplex $\Delta^{N-1}$ which is defined by embedding measure $\nu^{N-1}$ (defined on $\Delta^{N-2}$) into a higher dimensional simplex with extra coordinate $q_n=0$ as follows. For every Borel $B \subset \Delta^{N-1}$ let
$$\nu^{N-1}_{+n} (B) := \nu^{N-1}(B'),$$ where $B' = \{(q_1',\dots,q'_{n-1},q'_{n+1},\dots,q'_N) : (q_1',\dots,q'_{n-1},0,q'_{n+1},\dots,q'_N) \in B\}.$

\begin{algorithm}
\caption{$FEASIBLE((q_1,q_2,\dots,q_N),(t_1,t_2,\dots,t_N),\tau)$}\label{Alg_feasible}
\begin{algorithmic}
\Ensure $t_1<t_2<t_3\dots < t_N$, $q_1,\dots,q_N \ge 0$, $\tau$ is a measure on $\mathbb{R}$.
\State $t_{N+1}\gets \infty$
\State $q_{N+1}\gets 0$
\If{$\tau((-\infty,t_1)) > 0$}
	\Return (FALSE, $\nu = 0$)
\ElsIf{$N=0$} \Return (TRUE, $\nu=0$) 
\ElsIf{($\tau([t_1, t_N]) = 0$)} 
	\Return (TRUE, $\nu = 0$)

\Else \State Let $n$ be the smallest integer with $\tau([t_n, t_{n+1}])>0$
\If{$q_1 = 0$}
	\State $(val,\nu^{N-1}) \gets FEASIBLE((q_2,\dots,q_N),(t_2,\dots,t_N),\tau)$
	\State $\nu^N \gets \nu^{N-1}_{+1}$
	\State \Return $(val \mbox{ AND TRUE}, \nu^N)$	
\ElsIf {$q_{n+1}=0$ AND $N>1$}
	\State $(val,\nu^{N-1}) \gets FEASIBLE((q_1,\dots,q_{n},q_{n+2},\dots,q_N),(t_1,\dots,t_n,t_{n+2},\dots,t_N),\tau)$
	\State $\nu^N \gets \nu^{N-1}_{+(n+1)}$
	\State \Return $(val \mbox{ AND TRUE}, \nu^N)$	
\EndIf

\State Define, for $t \in [t_n,t_{n+1})$, $\alpha_1(t) := \frac{t_{n+1} - t}{t_{n+1} - t_1}$
\State Define, for $t \in [t_n,t_{n+1})$, $\alpha_{n+1}(t) := \frac {t-t_1}{t_{n+1}-t_1}$
\State It holds $t = \alpha_1(t) t_1 + \alpha_{n+1}(t) t_{n+1}$ for every $t \in [t_n,t_{n+1}]$. 
\State Define for $m\ge 0$ the function $m\mapsto t^m$ by $t^m := \inf\{t:\tau([t_n,t])>m\}$
\State Define for $m\ge 0$ the function $m\mapsto \tau^m$ by $\tau^m := \tau|_{[t_n,t^m)} + (m-\tau([t_n,t^m)))\delta_{t^m}$
\State Define for $m\ge 0$ the function $m\mapsto J_1(m)$ by $J_1(m):=\int_{\mathbb{R}} \alpha_1(t)d\tau^m$
\State Define for $m\ge 0$ the function $m\mapsto J_{n+1}(m)$ by $J_{n+1}(m):=\int_{\mathbb{R}} \alpha_{n+1}(t)d\tau^m$
\State $M \gets \tau([t_n,t_{n+1}])$
\State $M_1 \gets \sup\{m \le M: J_1(m)<q_1\}$
\State $M_{n+1} \gets \sup\{m \le M: J_{n+1}(m)<q_{n+1}\}$
\State $M_{min} \gets \min\{M_1,M_{n+1}\}$
\State $(val, \nu^{N}) \gets FEASIBLE((q_1-{J}_1(M_{min}),q_2,\dots,q_{n}, q_{n+1} - J_{n+1}(M_{min}), q_{n+2},\dots,q_N), (t_1,\dots, t_N), \tau - {\tau^{M_{min}}}) $

\State
$\nu(dq) \gets \int_{[t_n,t_{n+1}]}\delta_{(\alpha_1(t), \underbrace{0, \dots, 0}_{n-1}, \alpha_{n+1}(t), \underbrace{0, \dots, 0}_{N-(n+1)})}(dq) d{\tau^{M_{min}}}(t)$
\State which means that $\nu(B) = {\tau^{M_{min}}}(\{t: (\alpha_1(t), \underbrace{0, \dots, 0}_{n-1}, \alpha_{n+1}(t), \underbrace{0, \dots, 0}_{N-(n+1)}) \in B\})$ for Borel sets $B \subset \Delta^{N-1}$.
\State \Return $(val \mbox{ AND TRUE}, \nu^N + \nu)$
\EndIf
\end{algorithmic}
\end{algorithm}

\begin{proposition}[Stopping]
\label{Prop_Stopping}
Algorithm FEASIBLE always stops. 

\end{proposition}
\begin{proof}
Every call of $FEASIBLE((q_1,\dots,q_K), (t_1,\dots,t_K), \tau)$ satisfies:
\begin{itemize}
\item $t_1<t_2<\dots<t_K$,
\item $q_1,\dots,q_K \ge 0$,
\item $\sum_{k=1}^K q_k = \tau(\mathbb{R})$
\item $\sum_{k=1}^K q_k t_k = \int_{\mathbb{R}} t d\tau(t)$. 
\end{itemize}
Indeed, the first two conditions are easily verified, the third follows by the fact that $\alpha_1(t) + \alpha_{n+1}(t) = 1$ and hence $J_1(M_{min}) + J_{n+1}(M_{min}) = \tau^{M_{min}}(\mathbb{R})$. Finally, the last invariant is satisfied in the initial call by definition of travel time offer and in the recursive calls, we have, by integration of $t=\alpha_1(t)t_1 + \alpha_{n+1}(t)t_{n+1}$, 
\begin{equation*}
\int_{\mathbb{R}}  t d({\tau^{M_{min}}}(t)) = \int_{\mathbb{R}} (\alpha_1(t)t_1 + \alpha_{n+1}(t)t_{n+1}) d{\tau^{M_{min}}}(t) = J_1(M_{min})t_1 + J_{n+1}(M_{min}) t_{n+1}.
\end{equation*} 
Hence, the calls of FEASIBLE are correct. 
To prove that the algorithm will stop, we note that every call of FEASIBLE either decreases the number of points $t_n$ or nullifies the measure $\tau$ on one of the intervals $[t_n,t_{n+1}]$ or nullifies one of the coefficients $q_n$ (which in the next call results in decreasing the number of points $t_n$). As there are only finitely many points $t_n$ initially, then the algorithm has to stop. 
\end{proof}

\begin{definition}[Initial Section of measure]
\label{def_InitSec}
For every non-negative measure $\lambda$ on $\mathbb{R}$ and every $m \in [0,\lambda(\mathbb{R})]$ define the initial section $\lambda^m$ of $\lambda$ as follows.
\begin{eqnarray*}
x(m) &:=& \inf \{x \in \mathbb{R}: \lambda((-\infty,x])\ge m \},\\
\lambda^m &:=& \lambda|_{(-\infty,x(m))} + (m - \lambda((-\infty,m)))\delta_{x(m)}.
\end{eqnarray*}
\end{definition}

\begin{theorem}[Feasibility of offer profile]
\label{Th_feasibility}
Let $T^{CAV}$ be an offer profile subject to $(\bold{q},\bold{t}) = ((q_1,\dots,q_N),(t_1,\dots,t_N))$ and distributed according to $\tau$. Then the following conditions are equivalent:
\begin{enumerate}
\item[i)] $T^{CAV}$ is feasible.
\item[ii)] $FEASIBLE(\bold{q},\bold{t},\tau)$ returns $(TRUE,\nu)$.
\item[iii)] For every $m \in [0,q^{CAV}]$
\begin{equation}
\label{eq_Qm}
\mathbb{E}(Q^m) \le \mathbb{E}(\tau^m),
\end{equation}  
where $Q$ is the measure $Q = \sum_{r=1}^R q_r\delta_{t_r}$ and $\tau^m, Q^m$ are the initial sections of the respective measures defined in Definition \ref{def_InitSec}. 
\end{enumerate}

\end{theorem}

\begin{proof}
We prove the following implications:
\begin{enumerate}
\item[$ii)\to i)$] If $FEASIBLE(\bold{q},\bold{t},\tau)$ returns $(TRUE,\nu)$ then the obtained offer profile distribution $\nu$ allows for construction of an assignment plan inducing $T^{CAV}$. Indeed, 
by Algorithm \ref{Alg_feasible} we have:
\begin{equation*}
\nu = \int_{\mathbb{R}} \left( \sum_{k=1}^{K(t)} c_k(t)\delta_{\bold{a}^k(t)} \right) d\tau(t), 
\end{equation*}
where $\sum_{k=1}^{K(t)} c_k(t) = 1$ for every $t$ and $\bold{a}^k(t)$ belong to the simplex $\Delta^{R-1}$. $K(t)$ depends on the number of routings used to cover travel time $t$, and it may be greater than $1$ if $\tau$ has an atom at $t$ (i.e. when $\tau(\{t\})>0$) while it is equal $1$ otherwise. Setting $$\mu(i,r) := \sum_{k=1}^{K(T_i^{CAV})} c_k(T_i^{CAV})\left[\bold{a}^k(T_i^{CAV})\right]_r$$ we obtain the desired assignment plan. Indeed, $\sum_r \mu(i,r) = 1$ by the fact that $\bold{a}^k$'s coordinates add up to one, $$\int_I \mu(i,r)di =  \int_{\mathbb{R}} \sum_{k=1}^{K(t)} c_k(t)\left[\bold{a}^k(t)\right]_r d\tau(t) = \left[\int_{\mathbb{R}} \sum_{k=1}^{K(t)} c_k(t)\bold{a}^k(t) d\tau(t)\right]_r = q_r,$$ 
where the last inequality follows by the definitions of $J_1(m), J_{n+1}(m)$ and observation that $c_k(t)d\tau(t)$ corresponds to $\tau^m$ in a given call of Algorithm \ref{Alg_feasible} and finally $$\sum_r t_r^{CAV} \mu(i,r) =  \sum_{k=1}^{K(T_i^{CAV})} c_k(T_i^{CAV})\sum_r t_r^{CAV}\left[\bold{a}^k(T_i^{CAV})\right]_r = \sum_{k=1}^{K(T_i^{CAV})} c_k(T_i^{CAV})T_i^{CAV} = T_i^{CAV}.$$

\item[$iii)\to ii)$] If $FEASIBLE(\bold{q},\bold{t},\tau)$ returns $(FALSE,\nu)$ then at some nested call $FEASIBLE(\bold{\hat{q}},\bold{\hat{t}},\hat{\tau})$ we have $\hat{\tau}((-\infty,\hat{t}_1))>0$, where $\hat{\tau} = \tau - \tau^m$ for some $m$ and $\bold{\hat{q}}, \bold{\hat{t}}$ are the corresponding narrowed down sets of flows and travel times. Consequently, $\mathbb{E}({Q^m}) = \mathbb{E}({\tau^m})$ (by the invariants of Algorithm \ref{Alg_feasible}, see Proposition \ref{Prop_Stopping}) however $supp(\tau^{m+\epsilon} - \tau^m) \subset (-\infty,\hat{t_1})$ and $supp(Q^{m+\epsilon} - Q^m) \subset [\hat{t_1},\infty)$  for some $\epsilon>0$. As a result, $\mathbb{E}({Q^{m+\epsilon} - Q^m)} \ge \hat{t_1} > \mathbb{E}({\tau^{m+\epsilon} - \tau^m})$ and hence $\mathbb{E}({Q^{m+\epsilon})} > \mathbb{E}({\tau^{m+\epsilon}})$. 
\item [$i)\to iii)$] 
Fix $m \in [0,q^{CAV}]$. Then a feasible offer profile narrowed down to only drivers whose offers are covered by $\tau^m$ (possibly parts of drivers in the discrete case) is still feasible. These drivers are necessarily routed using some $\hat{Q} \le Q$ which satisfies $\mathbb{E}(\hat{Q}) = \mathbb{E}(\tau^m)$. Hence, $\mathbb{E}({Q^m})\le \mathbb{E}(\hat{Q}) = \mathbb{E}(\tau^m)$ where the first inequality is by the definition of initial section of a measure. This completes the proof. 
\end{enumerate}
\end{proof}

\begin{corollary}
Condition iii) in Theorem \ref{Th_feasibility} is a  useful practical criterion to verify feasibility of a given offer profile without the necessity to construct it explicitly, e.g. by running Algorithm \ref{Alg_feasible}. 
\end{corollary}

\begin{remark}
We can extend the notion of feasibility of an offer to offering travel times \emph{not exceeding} a given number $T_i^{CAV}$. With this definition, if $T^{CAV,A} \le T^{CAV,B}$ then if $T^{CAV,B}$ is feasible then also $T^{CAV,A}$ is feasible. 
\end{remark}

\begin{remark}
\label{Rem_feasibilityMixed}
Algorithm \ref{Alg_feasible} can be used to verify the feasibility of mixed routings with two or more different flow vectors $\bold{q^{CAV}}$. Namely, suppose $T^{CAV} = \sum_{m=1}^M p^m T^{CAV,m}$, where $\sum p^m = 1, p^m\ge 0$, are such that $T^{CAV,m}$ is an offer profile subject to $(\bold{q^{CAV,m}}, \bold{t^m})$. Then if $T^{CAV,m}$ is feasible for every $m$ then $T^{CAV}$ is also feasible with multi-day assignment scheme which can be obtained by applying multi-day assignment scheme corresponding to applying $q^{CAV,m}$ with probability $p^m$, every day.  
\end{remark}

\section{Fleet discount factors and optimal offer profiles}
In Sections \ref{Sec_OfferProfiles}-\ref{Sec_greedy} we studied the feasibility of given offer profiles and routings without delving into \emph {which} offer profile is the most convenient for the fleet operator. In this section, on the contrary, we discuss what offer profile may maximize market share, given a distribution of attitudes of drivers towards adoption of CAVs. 
We focus on the equilibrated full market share, i.e. we study the following problem, see (Q1) in Section \ref{Sec_ResQ}.
\\

\noindent{\bf Problem:} Is $u_i^{CAV} \le u_i^{HDV}$ possible for \emph{every} $i$ with some fixed feasible travel time offer $T_i^{CAV}$?

\noindent We assume that the utilities of using an HDV and CAV by driver $i$ are given by 
\begin{eqnarray}
u_i^{HDV} &=& t_{min}, \label{eq_ut1}\\
u_i^{CAV} &=& \gamma_i^F T_i^{CAV}, \label{eq_ut2}
\end{eqnarray}
where $t_{min}$ is the expected travel time of the fastest route and the discount factor $\gamma_i^F$ expresses the attitude (or relative value of time) of driver $i$ towards using CAVs. 

\begin{remark}
Utilities \eqref{eq_ut1}-\eqref{eq_ut2} typically contain random and other terms, see Appendix \ref{Sec_General}, however we disregard them here and focus on the leading term related to travel time and demonstrate the principal ensuing effects as a proof of concept. We also assume that in equilibrated settings, which we study, $t_{min}$ is known precisely to all drivers. 
\end{remark}

\noindent Let us now formally define discount factor profiles. 

\label{Sec_DiscountFactorProfiles}
\begin{definition}[{CAV discount factor profile}]
A \emph{CAV discount factor profile} is any measurable function $\gamma^F : I \to \mathbb{R}$, which relates the disutility $u_i$ experienced by driver $i$ when using a CAV to mean offered travel time by $u_i = \gamma^F_i T^{CAV}_i$.  
\end{definition}

\begin{remark}
\begin{enumerate}
\item[i)]
The default range of CAV discount factor profile is $[0, 1]$, which corresponds to decrease of travel time disutility due to using a CAV. However, there may exist users distasting/mistrusting the idea of CAV, for which $\gamma^F > 1$ (likely), and users who in fact \emph{benefit} from using a CAV to such a degree that $\gamma^F < 0$ (rather unlikely). The latter users should be routed via longest/most inconvenient routes. For brevity of arguments and calculations, however, we will disregard the possibility $\gamma^F \le 0$ and assume that the profile is qualitatively similar to the one in Fig. \ref{Fig_gammaF}, where $p(\gamma^F)$ is the probability density function of the distribution of $\gamma^F$, compare \cite{Correia}. 
\begin{figure}
\center
\includegraphics[scale=0.7]{"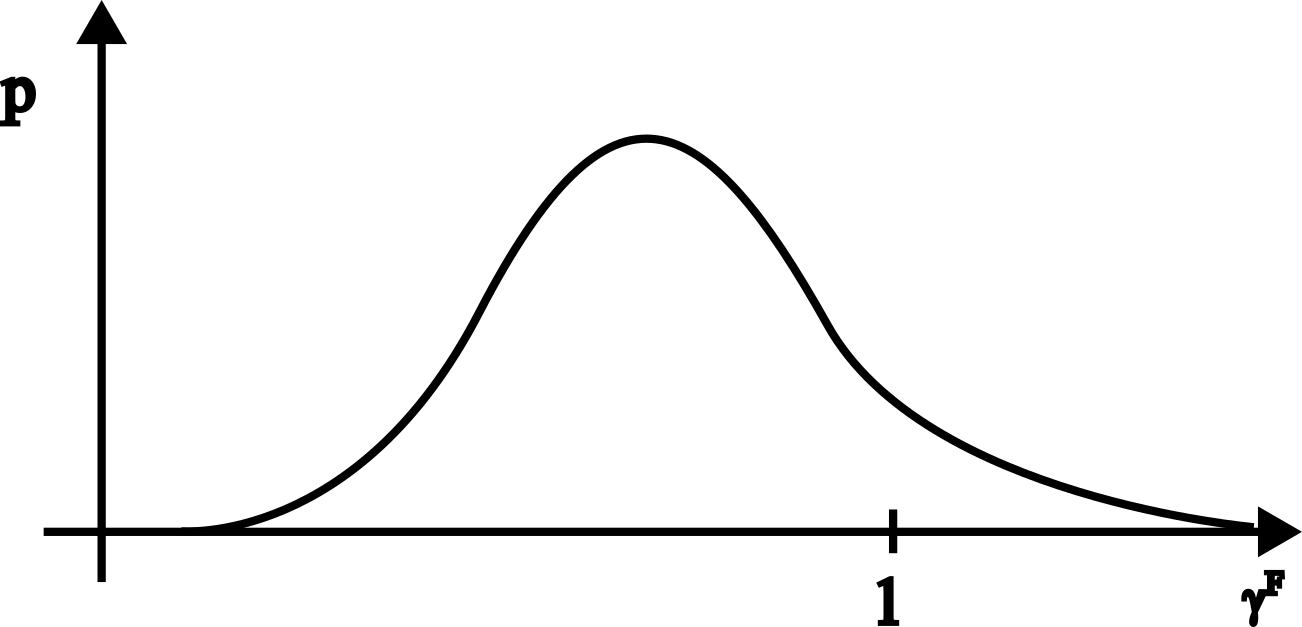"}
\caption{Typical assumed distribution of fleet discount factors in a population of drivers in a city. Most drivers have $\gamma^F<1$ which corresponds to lower value of time when using CAV and thus benefitting from using CAVs. $\gamma^F=1$ expresses indifference to using a CAV vs. HDV. Finally, some drivers, however, may distaste using CAVs which corresponds to the tail for $\gamma^F>1$.} \label{Fig_gammaF}
\end{figure}
\item[ii)] The linear dependence of utility on travel time in \eqref{eq_ut2} is an approximation and more generally one could consider $u_i^{CAV} = \gamma_i^F(T_i^{CAV})$ with $\gamma_i^F$ a function with non-constant slope depending on $T_i^{CAV}$.  This could encompass e.g. users which like long trips with a CAV but dislike short trips.  

\end{enumerate}
\end{remark}

\noindent The following result establishes simple conditions which relate the full market share achievability to the structure of routing $(\bold{q^{CAV}},\bold{t^{CAV}})$ and given discount factor profile $\gamma^F$. 

\begin{proposition}[Full market share necessary condition, deterministic routing]
\label{Prop_fullMS}
Let $\bold{q^{CAV}}$ be a fixed assignment of CAV flows to routes and $\bold{t^{CAV}} = (t_1^{CAV}, \dots, t_R^{CAV})$ the corresponding  travel time vector.  Let 
\begin{eqnarray}
t_{min}&:=& \min_r t_r^{CAV} \label{eq_tmin}\\
t_{max}&:=& \max_r t_r^{CAV} \label{eq_tmax}
\end{eqnarray}
be the minimum and maximum, respectively, travel time experienced on routes  and 
\begin{equation}
\overline{t^{CAV}} := \frac 1 {q^{CAV}} \sum_r q_r^{CAV}t_r^{CAV}
\label{eq_tcavmean}
\end{equation}
the mean CAV travel time, see \eqref{eq_compatibility}. Then 
\begin{enumerate}
\item [i)] \begin{equation}
\label{eq_necessary}
\overline{t^{CAV}} \le t_{min}\mathbb{E}(1/\gamma^F) 
\end{equation}
is a necessary condition for the existence of a feasible CAV offer profile subject to routing $(\bold{q^{CAV}},\bold{t^{CAV}})$ such that $u_i^{CAV}\le u_i^{HDV}$ for every $i$, meaning that no driver has an incentive to defect, i.e. switch from CAV to HDV. 
\item [ii)] If there are two routes only and $t_{min}/t_{max} \le \gamma_i^F \le 1$ for every user $i$, then \eqref{eq_necessary} is also a sufficient condition.
\end{enumerate}
\end{proposition}
\begin{proof}
i) The travel time each driver $i$ can be offered so that it does not defect satisfies $T_i^{CAV} \le t_{min}/\gamma_i^F$ by \eqref{eq_ut1}-\eqref{eq_ut2}. Integrating/summing over all drivers $i$ and applying \eqref{eq_compatibility}, we obtain \eqref{eq_necessary}. 

\noindent ii) Suppose first that \eqref{eq_necessary} is an equality. Then offering $T_i^{CAV} = t_{min}/\gamma_i^F$ results in $u_i^{CAV} = t_{min} = u_i^{HDV}$ for every user, which is sufficient to prevent switching from CAV to HDV. This $T_i^{CAV}$ is a valid offer profile as it satisfies \eqref{eq_compatibility} and so, by Proposition \ref{Prop_AssPlans_prop}ii, it is feasible. If \eqref{eq_necessary} is a strict inequality, 
one can easily find a (nonunique) profile $\tilde{T}^{CAV}$ such that $$t_{min} \le \tilde{T}_i^{CAV} \le t_{min}/\gamma_i^F$$ for $i \in I$ and $\tilde{T}_i^{CAV}$ satisfies \eqref{eq_compatibility}, i.e. $\frac 1 {|I|} \int_I {\tilde T}_i^{CAV} di = \overline{t^{CAV}}.$ Such a profile can be constructed e.g. as follows. Let $\alpha>1$ be the unique number satisfying $$t_{min} + \alpha( \overline{t^{CAV}} - t_{min}) = t_{min}\mathbb{E}(1/\gamma^F).$$ Then $$\tilde{T}_i^{CAV} = t_{min} + \frac 1 \alpha (t_{min}/\gamma_i^F - t_{min})$$ satisfies $t_{min} \le \tilde{T}_i^{CAV} \le t_{max}$ and yields the desired offer profile satisfying \eqref{eq_compatibility} which is robust against  defections from CAV to HDV. 
\end{proof}

\begin{remark}
\begin{enumerate}
\item [i)] The factors $\gamma_i^F$ are generally not known exactly to the fleet operator a priori and should be estimated. Bayesian methods seem to be a practical way to do it as the prior distribution could potentially be reasonably guessed for each individual driver and drivers' behaviour in day-to-day changing conditions is likely to enable estimation of a tight posterior. Devising an optimal scheme for this is future work. 
\item[ii)] Condition \eqref{eq_necessary} prescribes the space of potentially available solutions for equilibrium full market share. In particular, if system optimal assignment satisfies it, the fleet could be able to offer system optimal assignment with no incentive to defect by any driver.
\end{enumerate} 
\end{remark}

\begin{remark}
\label{Rem_gammasmall}
\noindent Proposition \ref{Prop_fullMS}ii requires $t_{min}/t_{max} \le \gamma_i^F \le 1$. This is to ensure that the offered travel time $T_i^{CAV} = t_{min}/\gamma_i^F$ is within the range $[t_{min},t_{max}]$. 
This condition can, however, be relaxed to  $\gamma_i^F \le 1$ by assigning drivers with $\gamma_i^F < t_{min}/t_{max}$ to longest routes \emph{before} verifying \eqref{eq_necessary} as follows. For two routes and $\gamma_i^F \le 1$ we write $I = I_1 \cup I_2$ where $I_1 = \{i: \gamma_i^F < t_{min}/t_{max}\}$ and $I_2 = \{i: t_{min}/t_{max} \le \gamma_i^F \le 1 \}$. Denote $r_{max} = \arg\max_r t_r^{CAV}, r_{min} = \arg\min_r t_r^{CAV}$ and $q_{max} = q_{r_{max}}, q_{min}=q_{r_{min}}$. There are two cases: 
\begin{itemize}
\item  $|I_1| \ge q_{max}$. Then assigning all drivers from  $I_2$ to $r_{min}$ (i.e. taking $\mu(i,r) = \delta_{rr_{min}}$, where $\delta$ is the Kronecker delta) and those from $I_2$ arbitrarily to the remaining slots so that the total assigned to routes is given by $\bold{q^{CAV}}$, yields an assignment plan and hence, by Proposition \ref{Prop_AssPlans_prop}i a feasible offer profile. 
\item $|I_1| < q_{max}$. Then considering $\bold{\tilde{q}^{CAV}} = (\tilde{q}_1^{CAV} - |I_1|,\tilde{q}_2^{CAV})$ if $r_{max} = 1$ or $\bold{\tilde{q}^{CAV}} = (\tilde{q}_1^{CAV},\tilde{q}_2^{CAV} - |I_1|)$ if $r_{max} = 2$    we obtain that a feasible offer profile exists if and only if \eqref{eq_necessary} is satisfied for $(\bold{\tilde{q}^{CAV}}, \bold{t^{CAV}})$ and driver space $I_2$. 
\end{itemize}

\end{remark}

\noindent The procedure described in Remark \ref{Rem_gammasmall} can be generalised to arbitrary number of routes using Algorithm \ref{Alg_feasible} (Algorithm FEASIBLE).

\begin{proposition}[Full market share sufficient condition, general case]
For an arbitrary number of parallel routes and a given discount factor profile $\gamma^F$ there exists a greedy algorithm verifying whether a couple $(\bold{q^{CAV}}, \bold{t^{CAV}})$ admits a feasible CAV travel time offer profile such that no driver has an incentive to defect to HDV ($u_i^{CAV} \le u_i^{HDV}$ for every $i$).
\end{proposition}

\begin{proof}
Put $T_i^{CAV} = \min\left(t_{min}/\gamma_i^F, t_{max}\right)$, where $t_{min}, t_{max}$ are defined in \eqref{eq_tmin}-\eqref{eq_tmax} and let $\overline{t^{CAV}}$ be defined as in \eqref{eq_tcavmean}. If 
\begin{itemize}
\item  $\mathbb{E}(T_i^{CAV}) < \overline{t^{CAV}} $ then there is no offer profile subject to $(\bold{q^{CAV}}, \bold{t^{CAV}})$ with no incentive to defect by Proposition \ref{Prop_fullMS}. 
\item $\mathbb{E}(T_i^{CAV}) = \overline{t^{CAV}} $ then use greedy Algorithm \ref{Alg_feasible} or Theorem \ref{Th_feasibility} to decide.
\item $\mathbb{E}(T_i^{CAV}) > \overline{t^{CAV}} $  
then modify Algorithm \ref{Alg_feasible} as follows. Algorithm $\ref{Alg_feasible}$ returns (TRUE, $\nu=0$) if both $\tau((-\infty,t_1))=0$ and $\tau([t_1,t_N])=0$. This relies on the fact that these two equalities imply that also $\tau((t_N,\infty)) = 0$ by the invariants $\sum_{k=1}^K q_kt_k = \int_{\mathbb{R}} t d\tau(t)$ and $\sum_{k=1}^K q_k = \tau(\mathbb{R})$, which in turn require $\mathbb{E}(T_i^{CAV}) = \overline{t^{CAV}}$. To obtain an algorithm that works for $\mathbb{E}(T_i^{CAV}) > \overline{t^{CAV}}$, we need to modify Algorithm \ref{Alg_feasible} by returning (TRUE, $\nu = \sum_k q_k \delta_{t_k}$) if $\tau((-\infty,t_1))=0$ and $\tau([t_1,t_N])=0$, bearing in mind that this part of $\nu$ assigns the drivers corresponding to $\tau|_{(t_N,\infty)}$ to the still available routes, when reconstructing $\mu$. We leave the (not difficult) proof that this modified algorithm returns (TRUE, $\nu$) if and only if there exists a feasible offer profile $\tilde{T}^{CAV} \le T^{CAV}$ subject to $(\bold{q^{CAV}}, \bold{t^{CAV}})$ to the reader.

\end{itemize} 

\end{proof}

\section{Experimental results}
\label{Sec_illustrative}

In this section we present a series of four illustrative scenarios, which address the following questions, see (Q1), (Q3) in Section \ref{Sec_ResQ}.
\begin{enumerate}
\item Can the fleet operator achieve full market penetration under various constraints on available routings?
\item What network conditions result from  market share maximization strategies? 
\end{enumerate}

\noindent We assume, following Section \ref{Sec_DiscountFactorProfiles}, that every driver $i$ is characterized by a fleet discount factor $\gamma_i^F$, which corresponds to the relative value of time of using a CAV as compared to using an HDV. The disutilities of using a CAV and using an HDV can then be expressed by \eqref{eq_ut1}-\eqref{eq_ut2}. 
Full market penetration is, in full CAV market share equilibrium equivalent to every driver $i$ being committed to using CAV, i.e. $u_i^{CAV} \le u_i^{HDV}$ for every $i$, which is equivalent to

\begin{equation}
\label{eq_onlyCAV}
\gamma_i^F T_i^{CAV} \le t_{min}
\end{equation}
for every $i$, see Section \ref{Sec_DiscountFactorProfiles}. Solving this inequality for various distributions $\gamma^F$ and assignments of CAV flows to routes $\bold{q^{CAV}}$ we distinguish four cases:
\begin{itemize}
\item In Section \ref{ex_1} we discuss symmetric offer profiles where every driver is offered the same mean travel time, compare Example \ref{Ex_sym}.  
\item Section \ref{ex_2} discusses a deterministic routing with travel time offer profile individually tailored to drivers, see Proposition \ref{Prop_fullMS}.
\item Section \ref{ex_3} demonstrates that mixed routings, see Remark \ref{Rem_feasibilityMixed}, are strictly superior to deterministic routings and shows the practical implications thereof.
\item In Section \ref{ex_4} we show how the full market share can be achieved dynamically starting from HDV-only Wardrop User Equilibrium, by application of a mixture of deterministic and mixed stochastic CAV strategies (routings). 
\end{itemize}

Let us point out that accross the four scenarios we use simple affine delay functions to model travel time on routes, which however do not restrict the generality of our considerations, as the only property that we use is the fact that they are strictly increasing, which is also the case for the more standard quartic BPR \cite{BPR} functions. In fact, we conjecture that the same effects will be present in more realistic simulations and even in monotonicity-violating real traffic systems.

\subsection{Symmetric offer profile}
\label{ex_1}
If $\bold{q^{CAV}}$ corresponds to system optimum and all drivers are offered the same mean travel time $T_i^{CAV} = \overline{T^{CAV,SO}}$ then  \eqref{eq_onlyCAV} yields $$\gamma_i^F \le T^{CAV,SO}_{min}/\overline{T^{CAV,SO}},$$ where $T^{CAV,SO}_{min}$ is the travel time of the fastest route at System Optimum. Satisfaction of this inequality heavily dependens on the structure of system optimal flows. Consider, namely, a system with two routes $A, B$ such that $t_A = 1 + 2q_A$, $t_B = 2+q_B$ and continuous routing of CAVs ${\bf q^{CAV}} = (q_A, q_B)$ such that  $q_A + q_B = 1$. Let $I = [0,1]$ be the indices of drivers. Simple computations yield:
\begin{itemize}
\item Wardrop Equilibrium: $q_A = \frac 2 3, q_B = \frac 1 3$, $t_A = t_B = {2 \frac 1 3}$. 
\item System Optimum: $q_A = q_B = \frac 1 2, t_A = {\bf 2}, t_B = {\bf 2 \frac 1 2}$. 
\item  $\overline{T^{CAV,SO}} = {\bf 2 \frac 1 4}$, $T^{CAV,SO}_{min} = {\bf 2}$, $T^{CAV,SO}_{min}/\overline{T^{CAV,SO}} = {\bf 8/9}.$
\end{itemize}
Consequently, if $\gamma_i^F < 8/9$ then for $T_i^{CAV} = 2\frac 1 4$ for every $i \in I$ no driver has an incentive to defect to using HDV and driving via the faster route $A$. We note that this condition on $\gamma^F$ is rather restrictive and one can usually not expect $\gamma^F$ to be bounded by a constant less than $1$ in a real system. We note also that the affine route delay functions were chosen for simplicity and any increasing delay function, e.g. BPR, would produce qualitatively similar outcomes.

\subsection{Tailored offer profile}

In Section \ref{ex_1} we assumed that CAVs offer the same mean travel time to \emph{all} the drivers. However, the travel time offers can be tailored to individual drivers provided the discount factors are known to the fleet controller. The next scenario shows how the system optimal routing can overcome the restrictive bound on $\gamma^F$ by individualizing offers, see Proposition \ref{Prop_fullMS} for a general result. 

\label{ex_2}
Consider the same road network as in Section \ref{ex_1}. Let $\gamma_i^F = 1.0$ for $50\%$ of drivers and $\gamma_i^F \le 0.8$ for the remaining drivers. Offering the drivers with $\gamma^F = 1$ always route $A$ and the drivers with $\gamma^F \le 0.8$ always route $B$ results in system optimal flows on routes $q_A = q_B = 0.5$. Moreover, no CAV user has an incentive to defect to HDV as the drivers on route $A$ are already using always the faster route and for them $u_i^{CAV} = t_{min} = u_i^{HDV}$ while the remaining drivers have $u_i^{HDV} = 2$ and $u_i^{CAV} \le 0.8 * 2.5 = 2.0 \le u_i^{HDV}$.  Finally, we note that average $1/\gamma^F$ is estimated by $$\mathbb{E}(1/\gamma^F) \ge 0.5 * (1/1.0) + 0.5*(1/0.8) = 1/2 + 5/8 = 9/8 =  \overline{T^{CAV,SO}}/T^{CAV,SO}_{min},$$ which reflects Proposition \ref{Prop_fullMS}.

Hence, the system optimal or any other fixed routing may make users not defect to HDV even for $\gamma^F$ up to $1$. Nevertheless, whether this is possible or not depends on the distribution of $\gamma^F$ and properties of the road system. And, clearly, drivers with $\gamma^F > 1$ cannot be convinced to remain part of the fleet of CAVs as routing them via the faster alternative still yields $u_i^{CAV}>u_i^{HDV}$. Can the fleet do any better? Section \ref{ex_3} shows that in many cases there exists a routing such that even users with $\gamma^F > 1$ have no incentive to defect to HDV. The existence of this routing relies on stripping the HDVs of the knowledge which route is the fastest on any given day. 

\subsection{Mixed routing}
\label{ex_3}
Assume there are two equivalent routes $A$ and $B$ with delay functions $t_A(q) = t_B(q) = 1 + q$. Suppose that $10\%$ of drivers have a discount factor $\gamma^F = 1.3$ and $90\%$ of drivers have $\gamma^F = 0.7$. Then no deterministic (even individually tailored) routing and travel time offer will ever convince the drivers with $\gamma^F =1.3$ to use a CAV as these drivers can simply defect and use the faster route. Suppose, however, we admit mixed routings of the form 
$$\bold{q^{CAV}} = \begin{bmatrix}
(q_A^1, q_B^1) & p\\
(q_A^2, q_B^2) & 1-p
\end{bmatrix},$$
which means that route flows $(q_A^1,q_B^1)$ are applied with probability $p$ and route flows $(q_A^2,q_B^2)$ with probability $1-p$ for $0\le p \le 1$. With this notation, consider routing $$\bold{q^{CAV}} = \begin{bmatrix}
(0.9, 0.1) & 0.5\\
(0.1, 0.9) & 0.5
\end{bmatrix},$$
i.e. $90\%$ of vehicles routed via A and $10\%$ via B with probability $0.5$ and $90\%$ of vehicles routed via B and $10\%$ via A with probability $0.5$, see Fig. \ref{Fig_MixedRout}. Moreover, assume that drivers with $\gamma^F = 1.3$ are offered to always be routed via the less congested alternative. For these drivers we have $u_i^{CAV} = 1.3(1+0.1) = 1.43$. The remaining drivers are always routed via the congested alternative and  have $u_i^{CAV} = 0.7(1+0.9) = 1.33$. The disutilities of using an HDV are, for all users, given by \emph{expected} utilities as HDVs have no access to information which route will be congested on which day, 
which can be computed as follows. Both routes have $50\%$ probability of having travel times $1.9$ and $50\%$ probability of having travel time $1.1$ and so $t_{min}$ is the expected travel time via any of them, given by $u_i^{HDV} =  t_{min} = 0.5*1.9 + 0.5*1.1 = 1.5 $. As for every $i$ we have $u_i^{CAV} < u_i^{HDV}$, this mixed routing with CAVs only is robust against defections to HDV. Note that this mixed routing is costlier to the fleet controller than the system optimal routing. Nevertheless, the system optimal routing splitting the flow equally (or any other deterministic routing), will never convince the reluctant group to use a CAV. 

\begin{figure}
\centering
\includegraphics[scale=0.6]{"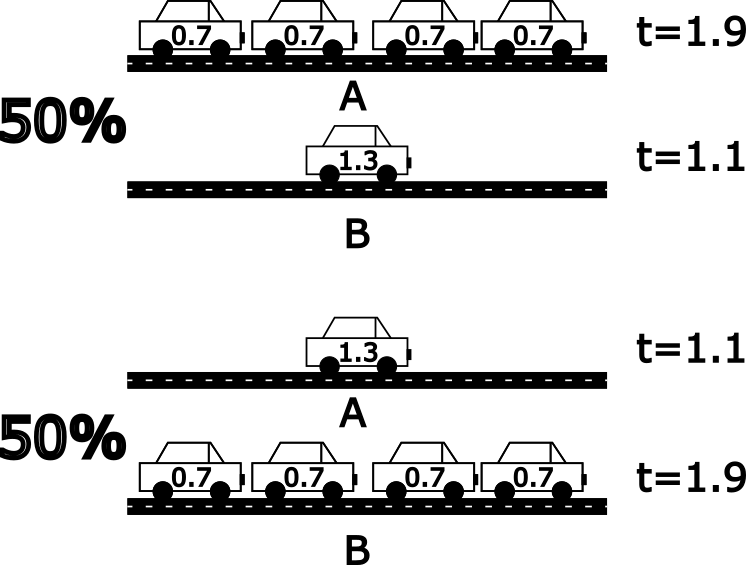"}
\caption{Mixed CAV routing in a system with two equivalent routes and a human driver population where $90\%$ are enthusiastic towards AVs (discount factor $0.7$) and $10\%$ dislike AVs (discount factor $1.3$). Routing the reluctant users always via the uncongested alternative prevents them from defecting to using HDV, however only when HDVs have no information which route will be congested on a given day. This can be achieved by a mixed routing such that there is $50\%$ probability that $90\%$ of vehicles are routed via the upper alternative and $10\%$ via the lower alternative and $50\%$ probability that $10\%$ of vehicles are routed via the upper alternative and $90\%$ via the lower alternative. This results in HDV utilities being equal to the expected travel time on any of the routes, as choosing the faster route on a specific day is now unavailable due to lack of information which was available when the driver was part of the CAV fleet.} 
\label{Fig_MixedRout}
\end{figure}

\begin{openproblem}
One may ask whether routing made up of exactly two strategies is the most cost-efficient way of increasing the expected travel times on routes for HDVs. A strong indication that this is indeed so when delay functions are strictly convex is Proposition 50 in \cite{JamrozIFA}. The general case with heterogenous human population is left for further research.  
\end{openproblem}

\subsection{Dynamic routing}
\label{ex_4}
In Section \ref{ex_3} we discussed the stability of $100\%$ CAV market share via mixed routing. Here, we show how the $100\%$ market share could be achieved in the dynamics setting. 

\begin{enumerate}[leftmargin = 1.5cm]
\item[Stage 0:] The system, composed of HDVs only, is in Wardrop Equilibrium. The drivers exhibit different attitudes to CAVs with $90\%$ enthusiastic, $\gamma_i^F = 0.7$, and $10\%$ reluctant with $\gamma_i^F = 1.3$. At some point the fleet operator enters the market. 
\item [Stage 1:] Fleet offers deterministic routing which mimics the choices of HDVs. All drivers with $\gamma^F<1$ are expected to join sooner or later, obtaining greater utility (lower disutility). However, all the remaining drivers are expected to stick to HDVs, resulting in $90\%$ market share. 
\item [Stage 2:] Fleet routes its $90\%$ market share by a Stackelberg strategy:
$\bold{q^{CAV}} = \begin{bmatrix}
(0.9, 0.0) & 0.5\\
(0.0, 0.9) & 0.5
\end{bmatrix}.$
After some time needed for adaptation the remaining drivers split $(0.05, 0.05)$ achieving User Equilibrium conditional on the fleet routing, with equal expected travel times on both routes. The drivers in the fleet still have no incentive to defect as their disutilities are given by  $u_i^{CAV} = 0.7*1.95 = 1.365$, while $u_i^{HDV} = 0.5(1.05 + 1.95) = 1.5$.
\item [Stage 3:] Fleet offers travel time not worse than $1.1$ to remaining HDVs with $\gamma^F = 1.3$ resulting in offered disutility not worse than $1.43$. Gradually, all the remaining HDVs migrate to the fleet, obtaining initial disutility $1.3* 1.05 = 1.365$ which gradually grows to $1.43$. Simultaneously, the remaing drivers' CAV utility decreases from $1.365$ to $1.33$, while HDV utility of all drivers remains constant equal $1.5$.  
\item [Stage 4:] Having $100\%$ market share, the fleet can optimize its cost by solving the cost optimization problem in the space of individualized mixed routings with the constraint that no one has an incentive to defect to HDV. 
\end{enumerate}
Let us remark that this case study neglects the temporal aspect of the evolution, which must be based on realistic CAV vs. HDV switching models, compare Section \ref{Sec_switching} and requires further research. Similarly, the solution to the optimization problem in Stage 4, remains open. 

\subsection{Summary}
To summarize the considered scenarios we note that in general CAVs have a multiobjective optimization problem at hand, i.e. on the one hand they aim to increase the market share, on the other however, they want to reduce their (average) cost. If all users have $\gamma_i^F = 1$ then the only way that CAVs can achieve full market penetration with deterministic strategies at equilibrium is to route according to Wardrop User Equilibrium. Using system optimal routing is bound to fail, unless e.g. people can be convinced that defection from the routing scheme means return to User Equilibrium as in \cite{Hoffmann}.  Now, suppose that at equilibrium CAVs propose individually tailored routing which 
achieves average travel time $\overline{T^{CAV}}$, 
however this is not enough to achieve full market penetration. The reasons can be as follows:
\begin{enumerate}
\item [i)] Even though $\overline{T^{CAV}} \le T^{CAV}_{min} \mathbb{E}(1/\gamma^F)$ for $(\bold{q^{CAV}},\bold{t^{CAV}})$, the offer profile cannot be realized as an assignment scheme, see Theorem \ref{Th_feasibility}. 
\item[ii)] $\overline{T^{CAV}} > T^{CAV}_{min} \mathbb{E}(1/\gamma^F)$ for given (e.g. system optimal) $(\bold{q^{CAV}},\bold{t^{CAV}})$. Then perhaps a different assignment, e.g. constrained system optimal in place of system optimal, should be considered. 
\item[iii)] $\overline{T^{CAV}} > T^{CAV}_{min} \mathbb{E}(1/\gamma^F)$ for \emph{any} proposed $100\%$ market share fleet assignment. Then any deterministic routing will fail and mixed routings should be considered. However, this case occurs only if some drivers have $\gamma_i^F > 1$, since otherwise the Wardrop UE can serve as a routing, for which defection to HDV does not improve travel disutility. 
\item[iv)] The mean travel times offered to users have been figured out wrongly (e.g. fleet controller misestimated $\gamma_i^F$ of users) so that even though $\overline{T^{CAV}} \le T^{CAV}_{min} \mathbb{E}(1/\gamma^F)$ holds and the offer profile is feasible for the estimated $\gamma^F$, some users are offered travel times rendering their CAV disutilities higher than HDV disutilities ($u_i^{CAV} > u_i^{HDV}$) if the actual $\gamma^F$ is used. 
\item[v)] $\gamma_i^F \ge 1$ for some users $i$ and mixed routing may be necessary to increase the market share.
\item[vi)] Market share cannot be $100\%$ even if mixed routings are considered. Then perhaps dynamic mixed routing, which is out of scope of this paper, could be considered. 
\item[vii)] The human behaviour model \eqref{eq_ut1}-\eqref{eq_ut2} is not accurate enough. 

\end{enumerate}

\section{Discussion}
\label{Sec_Discussion}
In this paper we:
\begin{itemize}

\item Developed a mathematical framework for studying CAV fleet individualized routing offers and assignment plans and showed how they can be implemented as actual assignment of drivers to routes (Section \ref{Sec_OfferProfiles}).

\item Proposed a greedy algorithm (Algorithm \ref{Alg_feasible}) and a simplified criterion (Theorem \ref{Th_feasibility}) for verifying whether a given offer profile is feasible. The proposed algorithm is constructive and it returns an assignment plan if the offer profile is feasible. We proved that the algorithm is correct and always stops (Section \ref{Sec_greedy}).   

\item Provided necessary and sufficient conditions for the existence of $100\%$ market share routing of CAVs when drivers exhibit heterogeneous attitudes towards adoptions of CAVs (Section \ref{Sec_DiscountFactorProfiles}).  

\item Demonstrated that highly uncertain travel times on routes distributed as in the right panel of Fig. \ref{Fig_traveltime}, may become common if there is a group CAV player maximizing market share via routing (Section \ref{Sec_illustrative}). Such distribution of travel times, resulting from mixed fleet strategies, translates to high day-to-day variability and unpredictability of travel times bringing about deterioration of driving conditions which is likely to be unacceptable to citizens. Some of the consequences of this atypical bimodal  distribution, such as dramatic schedule shifts of time critical trips, are highlighted in Appendix \ref{Sec_General}. 

\end{itemize}

The results provide strong evidence that collective routing in a free CAV market may be detrimental to the society and policy-makers are urged to bear this risk in mind and potentially regulate the market. One tool which may be useful to achieve it could be collective routing detection addressed in \cite{JamrozIFA}.

This paper opens a new area of research and there is a myriad of directions to follow up. Below we list the most interesting theoretical and practical issues on top of those mentioned in the text, the future research should focus on: 
\begin{itemize}
\item Developing behavioural models of adoption of CAVs and switching between HDV and CAV based on extension of discount factor idea. Developing sound behavioural models of route choice under very high bimodal day-to-day uncertainty.
   
\item Developing efficient routing algorithms for CAV operators, in particular efficient estimation of fleet discount factors.  This includes developing algorithms for optimal deterministic fleet routing using the feasibility-checking algorithm in Section \ref{Sec_greedy} and 
developing algorithms for optimal mixed Stackelberg routings.
\item Developing algorithms for optimizing the traffic (or fleet cost) with $100\%$ market share of CAVs, i.e. finding the least costly mixed strategy (if necessary) which still makes users not defect to HDV. 

\item Developing a general game-theoretic setting of games of one large variable-size player vs. many small players in the context of urban traffic routing games.  
\item Characterizing the Fleet-Human Equilibria in the general case. 
\item Including the dependence of travel time on the departure time and building a framework for inclusion of within-day rescheduling.
\item Extending all the results to systems with routes which are not parallel. 
\end{itemize}

\section{Acknowledgement}
This work was supported by the European Union within the Horizon Europe Framework Programme (ERC Starting Grant COeXISTENCE no. 101075838). Views and opinions expressed are however those of the authors only and do not necessarily reflect those of the European Union or the European Research Council Executive Agency. Neither the European Union nor the granting authority can be held responsible for them.

\appendix

\section{More general model formulation and schedule shift vs. uncertainty of travel times}
\label{Sec_General}
In this section we formulate a general model of CAV/HDV disutility in the fleet-independent driving choice context and show that mixed routing results, unlike in traditional day-to-day dynamics contexts, in dramatic schedule shifts of human drivers. 

To formulate the model, suppose there is one OD pair and $R$ parallel routes. Assume that $q = q^{HDV} + q^{CAV}$ is the total flow/number of vehicles, with $q^{CAV} = \alpha q$ fleet vehicles and $q^{HDV} = (1-\alpha)q$ human drivers. The fraction (market share) $\alpha$ is not fixed but can change depending on the utilities to the drivers of using an HDV or CAV. Assume that human drivers' value of time is equal $\beta_i$ and human drivers have discount factors $\gamma_i^F>0$ which scale the time-based disutility of using a CAV, see e.g. \cite{Correia}, by 
\begin{equation}
\label{eq_disutilityCAV}
u_i^{CAV} = u_i^{CAV, 0} + \gamma_i^F \beta_i T_i^{CAV} + u_i^{CAV, risk},
\end{equation}
while the disutility of using an HDV is given by
\begin{equation}
\label{eq_disutilityHDV}
u_i^{HDV} = u_i^{HDV,0} + \beta_i T_i^{HDV} + \epsilon_{ir(i)} + u_{ir(i)}^{HDV, risk}.
\end{equation}
Above, we assume that:
\begin{itemize}
\item The monetary cost of travel is proportional to travel time and incorporated into $\beta_i$ or negligible. 
\item $T_i^{CAV}$ is the average expected travel time the fleet of CAVs offers user $i$ on the given OD. 
\item $r(i)$ is the route HDV $i$ is dedicated to use in the near future. 
\item  $T_i^{HDV}$ is the long-term travel time of the most convenient route $r(i)$, as predicted/expected by driver $i$.  In more complex future scenarios, which we omit in the presentation below, $T_i^{HDV}$ along with $\epsilon_{ir(i)}$ and $u_{ir(i)}^{HDV, risk}$ may be the expectation over different routes if an HDV user predicts having to change routes regularly. Furthermore, $T_i^{HDV}$ may potentially be considerably misjudged based on previous experience resulting in self-confirming or bayesian equilibria. 

\item $\epsilon_{ir(i)}$ is the subjective preference related to route $r(i)$, as perceived by traveler $i$. 
\item $u_i^{CAV, 0}, u_i^{HDV, 0}$ are specific constants related to using HDV or CAV in general (e.g. prestige), and in the following they will be set to $0$, for simplicity. 
\item $u_{ir(i)}^{HDV, risk}$ is the term accounting for additional costs of travelling out of planned schedule, which we expect to be considerable, see below.
\item $u_i^{CAV, risk}$ is a risk term for using a CAV stemming from different route times. We will assume, for simplicity, that it vanishes, as the CAV is likely to be able to replace the office for most office-workers and variations of commute travel time may become negligible. 
\end{itemize}

The discount factors, $\gamma_i^F$ are assumed to be distributed according to a given probability measure $p$ on $[0,\infty)$. 
Behaviourally, the bulk of discount factors will be in the interval $[0.5, 1]$ which corresponds to a slight to moderate reduction of disutility when using a CAV compared to HDV, see \cite{Correia}, compare \cite{Harrison} for other aspects of CAV uptake.

\subsection{Example: Model with risk term}
In this section we discuss the outcomes of a simple yet realistic model of departure times adjustment by drivers facing high uncertainty. We assume the following disutilities:
\begin{eqnarray}
u_i^{CAV} &=& \gamma_i^F T_i^{CAV} \label{eq_disCAVmod2}\\
u_i^{HDV} &=& \min_r \left\{\mathbb{E}[T_{r}] + u_{ir}^{HDV, risk}\right\}\label{eq_disHDVmod2}
\end{eqnarray}
where
\begin{equation}
u_{ir}^{HDV,risk} = \min_{\rho} \mathbb{E} \left\{\theta_{LAP}  [T_r - \rho]^+ + \theta_{EAP} [\rho - T_r]^+ \right\}.
\label{eq_disHDVmod2Risk}
\end{equation}
Equation \eqref{eq_disCAVmod2} assumes that CAV users are insensitive to travel time variation and derive their disutility only from the average travel time offered to them. 
On the other hand, in \eqref{eq_disHDVmod2} we assume that human driven vehicles' disutilities are based on the route which is best (expressed as minimization over $r$ in \eqref{eq_disHDVmod2}) in terms of the sum of average travel time and variability. The variability is included via \eqref{eq_disHDVmod2Risk}, which expresses the process of departure time choice on every hypothetical alternative. It amounts to estimating the optimal amount of time $\rho$ to set off before the desired arrival time, see e.g. \cite{Noland, FosgerauKarlstrom}, compare \cite{Hollander} for the alternative approach using directly the variability (e.g. standard deviation or inter-quantile distance) in place of out-of-schedule penalties. We note that: 
\begin{itemize}
\item $T_r$ is assumed to be independent of $\rho$, which is a considerable simplification and does not hold in general. 

\item If $T_r$ is deterministic (constant), i.e. $T_r \thicksim \delta_{\mathbb{E}[T_r]}$, then setting $\rho = \mathbb{E}[T_r]$ we obtain $u_{ir}^{HDV,risk}=0$, which corresponds to departure such that arrival is always exactly on time and involves no risk. 
\item In general, there is a (perhaps non-unique) optimal $\rho$ which balances early arrival risk with late arrival risk and adds uncertainty-related cost of using an HDV on top of the cost based on average travel time. 
\item We assume that drivers are cognitively capable of estimating their $u_i^{HDV}$, however this is not essential, as they could also be using a dedicated app, which would do this comptuation for them if provided with value-of-time factors $\theta_{LAP}$ and $\theta_{EAP}$. 
\end{itemize}

\noindent In the following example we examine the potential consequences of bimodal day-to-day travel times distribution on routes, Fig. \ref{Fig_traveltime}, caused by mixed routing of the CAV fleet, see Section \ref{ex_3}. We discover that facing highly variable two-point day-to-day travel time distributions, drivers are likely to significantly shift their departure times (if they still want to use HDV), which is not the case for highly disruptive rare events.  

\begin{example}[Two-point travel time distribution]
\label{Ex_twopoint}
Assume $\theta_{LAP} = 2$, $\theta_{EAP} = 1$ in \eqref{eq_disHDVmod2Risk}, which seem to be reasonable choices, compare \cite{WatlingLAPUE}, and consider the following scenarios:

\begin{enumerate}
\item[i)] $T_r \thicksim 0.5 \delta_{1.1} + 0.5 \delta_{1.9}$ ($50\%$ probability of travel time $1.1$ and $50\%$ probability of travel time $1.9$) for $r=1,2$, which correspond to travel times resulting from mixed routing in Section \ref{ex_3}. Then, as both routes are exactly the same, and noting that optimal $\rho$ is in the interval $[1.1,1.9]$ we obtain, by \eqref{eq_disHDVmod2}-\eqref{eq_disHDVmod2Risk},
\begin{eqnarray*}
u_i^{HDV} = \mathbb{E}[T_r] + u_{ir}^{HDV,risk} &=& 1.5 + \min_{\rho} \mathbb{E} \left\{2  [T_r - \rho]^+ + [\rho - T_r]^+ \right\}\\ &=& 1.5 + 0.5\min_{\rho} \left\{2(1.9 - \rho) + (\rho - 1.1) \right\} = 1.5 + 0.4 = 1.9.
\end{eqnarray*}
Consequently, the optimal $\rho$ equals $1.9$ resulting in HDV disutility $1.9$, which means that, facing high uncertainty, it is best to significantly change (compared to departing the average travel time before the desired arrival) the schedule so as to never be late as {\bf being late on $50\%$ of days is not feasible,} see Fig. \ref{Fig_HumPersp}. If the early arrival value of time is equal $1$ (i.e. $\theta_{EAP}=1$) then the resulting disutility corresponds to the longest possible day-to-day travel time on routes ($1.9$ in our example). 

\begin{figure}[h!]
\centering
\includegraphics[scale=0.5]{"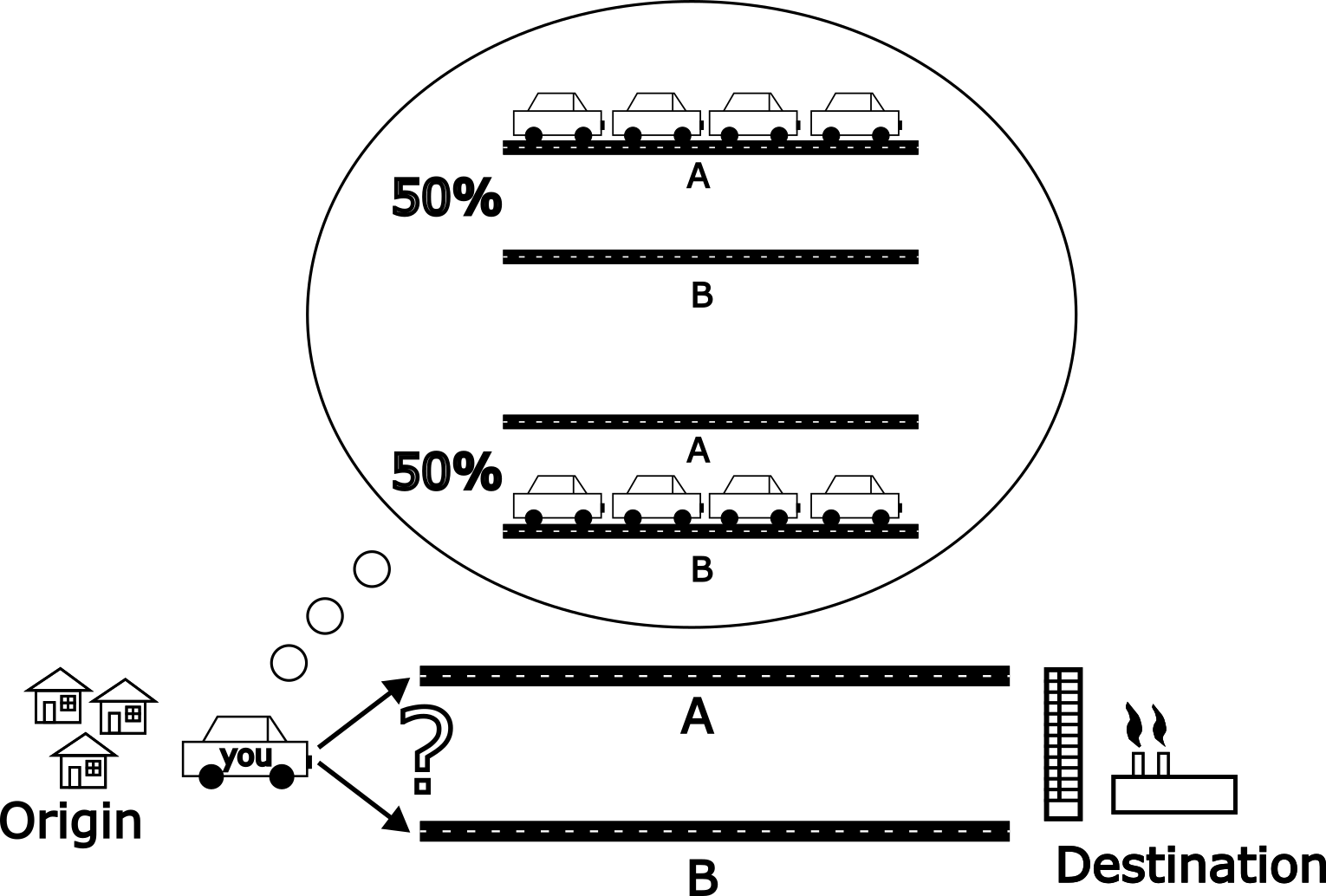"}
\caption{A human driver facing the route and departure time choice with high uncertainty of travel times, which assume only two possible considerably different values. We conjecture, based on the reasonable model we studied, that people will tend to dramatically shift their schedules so as to never arrive late. This, however comes at a cost of 'early arrival penalty' $50\%$ of time and so significantly increases their disutility of travelling. If this disutility is too high, they may want to switch to CAVs offering better experience, even if they are reluctant to adopt this new technology. }
\label{Fig_HumPersp}
\end{figure}

\item[ii)] The dramatic schedule shift does not occur when the risk is smaller. If $T_r \thicksim 0.9 \delta_{1.1} + 0.1 \delta_{1.9}$ then the optimal $\rho$ is equal $1.1$ resulting in disutuility $1.66$. 
\item[iii)] In general, for $T_r \thicksim (1-p)\delta_{T_{min}} + p\delta_{T_{max}}$, $r=1,2,\dots,R$ and  assuming the HDV utility given by \eqref{eq_disHDVmod2}-\eqref{eq_disHDVmod2Risk} we obtain 
\begin{equation*}
u_{ir}^{HDV,risk} = \min_{\rho}\left\{p\theta_{LAP} (T_{max} - \rho) + (1-p)\theta_{EAP} (\rho-T_{min}) \right\}
\end{equation*}
which is minimized for $\rho = T_{max}$ if $p\theta_{LAP} > (1-p)\theta_{EAP}$ and for $\rho = T_{min}$ otherwise. Consequently if for a given driver $$\theta_{LAP}/\theta_{EAP} > (1-p)/p$$ then the driver will set off $T_{max}$ before the desired arrival time and if $$\theta_{LAP}/\theta_{EAP} < (1-p)/p$$ then they will set off $T_{min}$ before the desired arrival time. In the corner case $\theta_{LAP}/\theta_{EAP} = (1-p)/p$ any $\rho \in [T_{min},T_{max}]$ will result in the same disutility. For the default $\theta_{LAP} = 2$, $\theta_{EAP} = 1$ we obtain  a {\bf threshold} probability $p = 1/3$ which results in drivers switching from setting off $T_{min}$ to $T_{max}$ before the desired arrival time. Note that this threshold probability is relatively high and therefore the effect of {\bf significant schedule change only occurs for high probability of disruptive events} and is specific to scenarios when some actor (like CAV fleet operator) deliberately introduces high variability of travel times. It does not occur in standard scenarios with low risk of disruptive events and drivers facing linear late arrival penalty, reasonable for late arrival at most kinds of workplaces.  
\end{enumerate}
\end{example}

Example \ref{Ex_twopoint} discussed the schedule change for drivers facing high uncertainty in the case of (typical for us in this paper) two point travel time distribution introduced deliberately by the fleet of CAVs. Proposition \ref{Prop_risk} discusses the more general situtation. 

\begin{figure}[h!]
\label{Fig_rhomin}
\centering
\includegraphics[scale=0.7]{"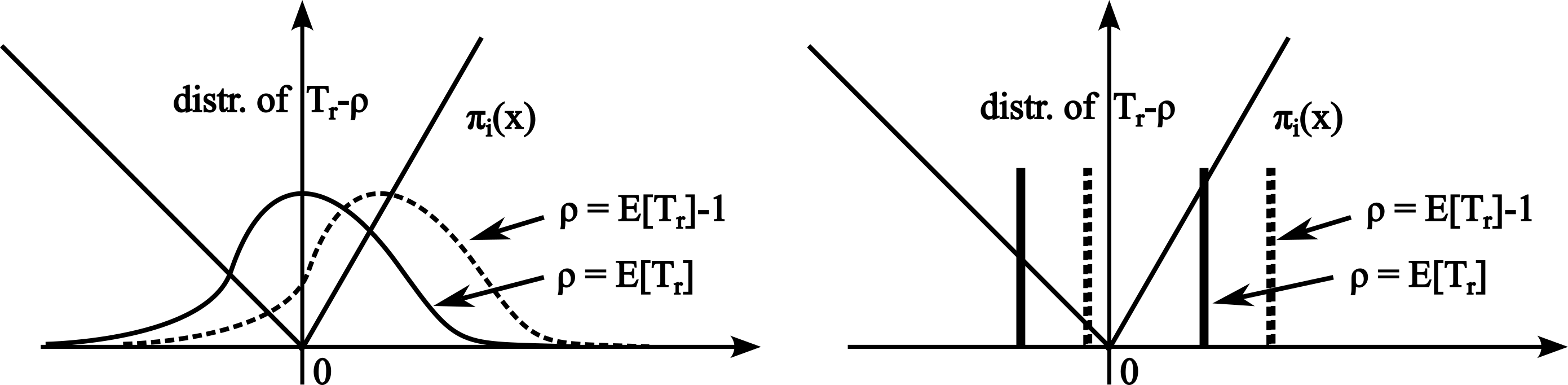"}
\caption{Visualization of function $\pi_i(x) = 2|x|^+ + |x|^-$ and distribution of $T_r - \rho$ in the continuous(left) and discrete(right) case for two values of $\rho$. The minimization $\min_{\rho} \mathbb{E}[\pi_i(T_r - \rho)]$ consists (in the continuous case) in finding the offset $\rho$ such that the integral $\int g(x + \rho)\pi_i(x)dx$ is minimized, where $g$ is the density of distribution of $T_r$ and hence $g(\cdot + \rho)$ is the density of distribution of $T_r - \rho$. If $T_r$ has a discrete distribution $w_1 \delta_{t_1} + w_2 \delta_{t_2}$ (in the right panel we have $w_1=w_2 = 0.5$) then the minimization consists in minimizing the weighted sum of the values of $\pi_i$, i.e. $\min_{\rho} \sum w_k \pi_i(t_k - \rho).$} 
\end{figure}

\begin{proposition}[General risk - schedule shift relation]
\label{Prop_risk}
Suppose the utility of HDV is given by
\begin{equation*}
u_i^{HDV} = \min_r \left\{\min_{\rho}\mathbb{E}[T_{r} + \pi_i(T_r-\rho)] \right\} = \min_r \left\{\mathbb{E}[T_r] +  \min_{\rho}\mathbb{E}[\pi_i(T_r-\rho)] \right\}
\end{equation*}
where $\pi_i$ is a convex non-negative out-of-schedule-arrival penalty function for driver $i$ satisfying $\pi_i(0)=0$ and $T_r$ is a distribution of travel times on route $r$ with bouned support. Then:
\begin{enumerate}
\item[i)] If $\pi_i$ is continuously differentiable and strictly convex, there exists a unique $\rho$ minimizing $\mathbb{E}[\pi_i(T_r-\rho)]$.
\item[ii)] If $\pi_i$ is not strictly convex then the minimizers may form an interval. 
\end{enumerate}
\end{proposition}
\begin{proof}
i) Denoting $\Pi(\rho):= \mathbb{E}[\pi_i(\rho-T_r)]$ we obtain that 
\begin{equation*}
\Pi'(\rho) = \mathbb{E}[\pi_i'(\rho-T_r)]
\end{equation*}
is strictly increasing. Clearly, $\Pi'(\rho)<0$ for $\rho$ small enough and $\Pi'(\rho)>0$ for $\rho$ large enough. Consequently, there exists a unique $\rho$ such that $\Pi'(\rho) = 0$, and $\rho$ is the minimizer of $\Pi$.  ii) Example \ref{Ex_twopoint}iii). 

\end{proof}

Returning to disutilities \eqref{eq_disCAVmod2}-\eqref{eq_disHDVmod2} we note that the disutility of using a CAV {\bf does not} include a risk term. This is due to the fact that, every day, driver $i$ is picked up at the departure time (varying day-to-day) such that he/she arrives at the destination exactly on time. Consequently, using an HDV involves not only unpredictability of travel times but also schedule adjustment, which 
makes using HDVs even more inconvenient compared to simple disutilities given by \eqref{eq_ut1}-\eqref{eq_ut2}. Note however, that this extra penalty plays a role {\bf only when the fleet of CAV uses mixed routing}. 

Finally, let us shortly dwell on the scenario when there is exogenous uncertainty involved. This uncertainty cannot be predicted by either HDVs or the fleet operator and comes in as additonal smoothing of travel time distribution. To be more precise, let $\phi_r^{\epsilon} \ge 0$ be small kernels accounting for small variations of travel time via different routes (note that here we exclude extreme events) suported in $[-\epsilon_r,\epsilon_r]$ such that $\int_{\mathbb{R}} \phi^{\epsilon_r}(x)dx = 1$. Then the travel times distribution is given by  
\begin{equation*}
T_r^{\epsilon_r} := T_r * \phi^{\epsilon_r}
\end{equation*}
where $*$ denotes convolution, and in the scenario considered in Example \ref{Ex_twopoint} resembles the right panel in Fig. \ref{Fig_traveltime}. Both HDVs and CAVs have to bear this additional risk, and the modified model \eqref{eq_disCAVmod2}-\eqref{eq_disHDVmod2Risk} may now look as follows:
 
\begin{eqnarray}
u_i^{CAV} &=& \gamma_i^F T_i^{CAV} + u_i^{CAV,risk} \label{eq_disCAVmod5}\\
u_i^{HDV} &=& \min_r \left\{\mathbb{E}[T_{r}] + u_{ir}^{HDV, risk}\right\}\label{eq_disHDVmod5}
\end{eqnarray}
where
\begin{eqnarray}
u_{ir}^{HDV,risk} &=& \min_{\rho} \mathbb{E} \left\{\theta_{LAP}  [T_r^{\epsilon_r} - \rho]^+ + \theta_{EAP} [\rho - T_r^{\epsilon_r}]^+ \right\},
\label{eq_disHDVmod5Risk}\\
u_i^{CAV,risk} &=& \sum_r \mu(i,r) \min_{\rho_r} \int_{\mathbb{R}}\left\{\theta_{LAP}  [\phi^{\epsilon_r}(x) - \rho_r]^+ + \theta_{EAP} [\rho_r - \phi^{\epsilon_r}(x)]^+ \right\}dx,
\label{eq_disCAVmod5Risk}
\end{eqnarray}
with $\mu(i,r)$ the proportion with which driver $i$ is routed via route $r$, see Section \ref{Sec_OfferProfiles}. Note that if all the routes have the same uncertainties $\phi^{\epsilon_r} = \phi^{\epsilon}$ then \eqref{eq_disCAVmod5Risk} reduces to 
\begin{equation*}
u_i^{CAV,risk} = \min_{\rho} \left\{\theta_{LAP}  [\phi^{\epsilon} - \rho]^+ + \theta_{EAP} [\rho - \phi^{\epsilon}]^+ \right\},
\end{equation*}
which is just a constant. We leave the analysis to further research, noting that other parts of utility, such as personal preferences or mode specific constants can be also incorporated into the model as per \eqref{eq_disutilityCAV}-\eqref{eq_disutilityHDV}. 
The proper inclusion of these and other aspects such as extreme disruptive events and dependence of $T_r$ on $\rho$ are left as an open problem.

\section{Routing via no more than two routes}
\label{Sec_proofs}
\begin{proposition}[2RMAX simplified assignment plan]
\label{Prop_2RMAX}
Let $\mu$ be an assignment plan generating the offer profile $T^{CAV}$. Let $\nu$ be the corresponding distribution on the unit simplex $\Delta^{R-1}$. Then there exists a (nonunique in general) 2RMAX assignment plan $\mu^R \sim \nu^R$ inducing the same travel time offer profile distribution $\tau$ such that $\nu^R$ is supported on $\{(q_1,\dots,q_R): \sum_{r=1}^R \bold{1}_{q_r > 0} \le 2\}$, i.e. on the set corresponding to routings with no more than two routes used. $\mu^R$ can be supported on an augmented space, which splits every (atomic or infinitesimal) driver $i$ into finitely many parts. Assignment $\mu$ can be recovered from $\mu^R$ by recombining split drivers $i$. 
\end{proposition}

\begin{proof}[Proof of Propositon \ref{Prop_2RMAX}]
For every $\bold{q}$ in the simplex $\Delta^{R-1}$ the function 2RMAX(1,R,$\bold{q}$,$\bold{t}$) defined in Algorithm \ref{Alg_2rmax} returns a finite collection of 2RMAX routings $\{\alpha^{r_1r_2}(\bold{q}) \delta_{\bold{q}^{r_1r_2}(\bold{q})}\}$ 
where $\bold{q}^{r_1r_2}(q)$ has non-zero entries at positions $r_1,r_2$ where each couple $r_1r_2$ is used at most once, $r_1 = r_2$ means that there is only one non-zero entry and most of the coefficients $\alpha^{r_1r_2}$ vanish. Let $I \times \{1,\dots,R\} \times \{1,\dots,R\}$ be the new augmented sample space with measure $\tilde{di}(A) = \int_{I}\sum_{(r_1,r_2)} \alpha^{r_1r_2}(\bold{q}(i)) \bold{1}_{(i,r_1,r_2) \in A} di$, for $A \subset I \times \{1,\dots,R\} \times \{1,\dots,R\}$ and $\mu^R(i,r_1,r_2,r) = \bold{q}^{r_1,r_2}(\bold{q}(i))_r$.
 We note that $\mu^R$ induces the same distribution $\tau$ as $\mu$ and the corresponding $\nu^R$ is supported on $\bold{q}$ with at most two non-zero coordinates.  Note also that Algorithm \ref{Alg_2rmax} is correct and always stops as every recursive call preserves the invariant $\overline{t}$ and every $NewEntry$ is a measure supported on $\{\bold{q}\}$ such that $\bold{q}\cdot \bold{t} = \overline{t}$.  
\end{proof}
\begin{algorithm}
\caption{2RMAX($r_{min}$, $r_{max}$, $\bold{c}$, $\bold{t}$) }\label{Alg_2rmax}
\begin{algorithmic}
\State $\overline{t}\gets \frac {1}{|\bold{c}|}\bold{c}\cdot \bold{t}$
\If{$r_{min} > r_{max}$} \Return $\emptyset$
\ElsIf{$c_{r_{min}} = 0$} \Return 2RMAX($r_{min}+1,r_{max},\bold{c},\bold{t}$) 
\ElsIf{$c_{r_{max}} = 0$} \Return 2RMAX($r_{min},r_{max}-1,\bold{c},\bold{t}$) 
\ElsIf{$r_{min}=r_{max}$} \Return $c_{r_{min}} \delta_{(0,\dots,0,\underset{\overset{\uparrow}{r_{min}}}{1},0,\dots,0)}$
\Else \State $\gamma \gets \frac {c_{r_{min}}t_{r_{min}} + c_{r_{max}}t_{r_{max}}}{c_{r_{min}}+c_{r_{max}}}$
\If {$\gamma>\overline{{t}}$} 
\State Let $0\le \tilde{c} < c_{r_{max}}$ satisfy $\frac {c_{r_{min}}t_{r_{min}} + \tilde{c}t_{r_{max}}}{c_{r_{min}}+\tilde{c}} = \overline{t}$ 
\State $NewEntry \gets  
\{(c_{r_{min}}+\tilde{c})\delta_{(0,\dots,0, \underset{\overset{\uparrow}{r_{min}}}{\frac {c_{r_{min}}}{c_{r_{min}} + \tilde{c}}}, 0,\dots,0,\underset{\overset{\uparrow}{r_{max}}}{\frac {\tilde{c}}{c_{r_{min}} + \tilde{c}}}, 0, \dots, 0    )}\}$
\State\Return 
$NewEntry \cup 2RMAX(r_{min}+1, r_{max}, \bold{c} - (0,\dots,0,c_{r_{min}},0,\dots,0,\tilde{c},0,\dots,0))$
\ElsIf {$\gamma<\overline{{t}}$} 
\State Let $0\le \tilde{c} < c_{r_{min}}$ satisfy $\frac {\tilde{c}t_{r_{min}} + c_{r_{max}}t_{r_{max}}}{c_{r_{max}}+\tilde{c}} = \overline{t}$ 
\State $NewEntry \gets  
\{(c_{r_{max}}+\tilde{c})\delta_{(0,\dots,0, \underset{\overset{\uparrow}{r_{min}}}{\frac {\tilde{c}}{c_{r_{max}} + \tilde{c}}}, 0,\dots,0,\underset{\overset{\uparrow}{r_{max}}}{\frac {c_{r_{max}}}{c_{r_{max}} + \tilde{c}}}, 0, \dots, 0    )}\}$
\State\Return 
$NewEntry \cup 2RMAX(r_{min}, r_{max}-1, \bold{c} - (0,\dots,0,\tilde{c},0,\dots,0,c_{r_{max}},0,\dots,0))$
\ElsIf {$\gamma=\overline{{t}}$} 
\State $NewEntry \gets  
\{(c_{r_{min}}+c_{r_{max}})\delta_{(0,\dots,0, \underset{\overset{\uparrow}{r_{min}}}{\frac {c_{r_{min}}}{c_{r_{min}} + c_{r_{max}}}}, 0,\dots,0,\underset{\overset{\uparrow}{r_{max}}}{\frac {c_{r_{max}}}{c_{r_{min}}+c_{r_{max}}}}, 0, \dots, 0    )}\}$
\State\Return 
$NewEntry \cup 2RMAX(r_{min}+1, r_{max}-1, \bold{c} - (0,\dots,0,c_{r_{min}},0,\dots,0,c_{r_{max}},0,\dots,0))$

\EndIf
\EndIf
\end{algorithmic}
\end{algorithm}

\begin{example}
\begin{enumerate}
\item[i)]
Let $\mu \sim \nu = \delta_{\left(\frac 1 4,\frac 1 2,\frac 1 4\right)}$ for $\bold{t} = (10,20,30)$ and $\bold{q} = (0.25,0.5,0.25)$. The same travel time distribution $\tau = \delta_{20}$ (every driver's average travel time equal $20$) is generated by $\mu^R \sim 0.5 \delta_{(0,1,0)} + 0.5\delta_{(\frac 1 2, 0, \frac 1 2)}$, where each infinitesimal driver uses at most two routes, compare Proposition \ref{Prop_AssPlans_prop}iii. Indeed, Algorithm \ref{Alg_2rmax} returns 
$$2RMAX(1,R,(0.25,0.5,0.25), (10,20,30)) = \{\alpha^{13}\delta_{\bold{q}^{13}(\bold{q})}, \alpha^{22}\delta_{\bold{q}^{22}(\bold{q})}\},$$
where $\bold{q}^{13} = (0.5,0,0.5), \bold{q}^{22} = (0,1,0)$, $\alpha^{13}=\alpha^{22}=0.5$. Suppose now there are $4$ drivers in the system, i.e. $I = \{1,2,3,4\}$ and the original assignment plan $$\mu = \begin{bmatrix}
0.25 & 0.5 & 0.25 \\
0.25 & 0.5 & 0.25 \\
0.25 & 0.5 & 0.25 \\
0.25 & 0.5 & 0.25 
\end{bmatrix}$$
The 2RMAX routings are supported on $\tilde{i} = I \times \{(1,3), (2,2)\}$ with the measure $\tilde{di}(A) = 0.5di(A \cap I \times \{(1,3)\}) + 0.5di(A \cap I \times \{(2,2)\}$ and 
$$\mu^R((i,1,3),r) = \begin{bmatrix}
0.5 & 0 & 0.5 \\
0.5 & 0 & 0.5 \\
0.5 & 0 & 0.5 \\
0.5 & 0 & 0.5 
\end{bmatrix}, \mu^R((i,2,2),r) = \begin{bmatrix}
0 & 1 & 0 \\
0 & 1 & 0 \\
0 & 1 & 0 \\
0 & 1 & 0 
\end{bmatrix}$$
which means that every driver $i$ is further split into two parts with equal weights $0.5$ and one half is routed according to $\mu^R((i,1,3),r)$ and the other according to $\mu^R((i,2,2),r)$. 
\item[ii)]
Let $\mu \sim \delta_{(\frac 1 4,\frac 1 4,\frac 1 4,\frac 1 4)}$ for $\bold{t} = (10,20,30,40)$ and $\bold{q} = (0.25,0.25,0.25,0.25)$ and $T^{CAV} \sim \tau = \delta_{25}$. Then $\mu^R_1 \sim 0.5\delta_{(\frac 1 2,0,0,\frac 1 2)}+0.5\delta_{(0,\frac 1 2,\frac 1 2,0)}$ is one 2RMAX plan generating $\tau$, while $\mu^R_2 \sim 1/3\delta_{(\frac 1 4,0,\frac 3 4,0)} + 1/3\delta_{(0,\frac 3 4,0,\frac 1 4)} + 1/3 \delta_{(\frac 1 2,0,0,\frac 1 2})$ is another. Recombining (i.e. integrating) $\mu^R$ along the additional dimensions we can recover $\mu$. 
\end{enumerate}
\end{example}

\end{document}